\documentclass[12pt]{amsart}

\usepackage{amssymb}
\usepackage{amsmath}
\usepackage{amsthm}
\usepackage{amsbsy}
\usepackage{amscd}
\usepackage{bm} 
\usepackage{graphicx}
\usepackage{xcolor}
\usepackage{verbatim}

\theoremstyle{plain}
\newtheorem{theorem}{Theorem}[section]
\newtheorem{lemma}[theorem]{Lemma}
\newtheorem{corollary}[theorem]{Corollary}
\newtheorem{proposition}[theorem]{Proposition}
\newtheorem{prop}[theorem]{Proposition}

\theoremstyle{definition}
\newtheorem{definition}[theorem]{Definition}

\theoremstyle{remark}
\newtheorem{remark}[theorem]{Remark}

\newcommand{\R}{\mathbb{R}}
\newcommand{\C}{\mathbb{C}}

\newcommand{\cK}{\mathcal{K}} 
\newcommand{\D}{\mathcal{D}} 
\newcommand{\I}{\mathcal{I}}

\newcommand{\eps}{\epsilon}
 
\newcommand{\p}{\partial}
\newcommand{\wt}{\widetilde}
\newcommand{\wh}{\widehat}
\newcommand{\Int}{\mathrm{Int\,}}
\newcommand{\ind}{\mathrm{ind}}
\newcommand{\Ker}{\mathrm{Ker\,}}
\newcommand{\dist}{\mathrm{dist}}
\newcommand{\Id}{\mathrm{Id}}
 \newcommand{\Op}{ {\mathcal O}{\it p}\,}
  \newcommand{\ol}{\overline}
   \newcommand{\om}{\omega}
   \newcommand{\corner}{\mathrm{Corner}}
      \newcommand{\Coker}{\mathrm{Coker\,}}
        \newcommand{\Skel}{\mathrm{Skel}}
          \newcommand{\st}{\mathrm{st}}

\begin{document}

\dedicatory{To Norbert A'Campo with admiration}  \title{Honda-Huang's work on contact convexity revisited}
\author{Yakov Eliashberg }
\address{Department of Mathematics\\Stanford University \\ Stanford, CA 94305 USA}
\email{eliash@stanford.edu}
\thanks{YE was partially supported by NSF grant DMS-2104473 and a grant from the Institute
for Advanced Study}

\author{Dishant Pancholi}
\address{The Institute of Mathematical Sciences\\
Taramani,
Chennai 600 113,
Tamil Nadu, India}
 \email{dishant@imsc.res.in} 
 \thanks{DP was partially supported by NSF grant DMS-1926686 during his stay at the Institute for Advanced Study}	 
\subjclass{Primary: 57R17; Secondary: 37C10}

\begin{abstract}
Following the overall strategy  of the paper \cite{HH} by  Ko Honda and Yang  Huang on contact convexity in high dimensions,  we present  a simplified proof of their main result.
\end{abstract}
\maketitle

\section{Introduction}

A hypersurface in a contact manifold is said to be convex  if it admits  a transverse contact vector field  (see Section \ref{sec:conv} below for precise definitions). The central   result of  the article \emph{``Convex hypersurfaces in contact topology"} by  Ko Honda and Yang Huang is  the following:
\begin{theorem}[Ko Honda and Yang Huang,   \cite{HH}]\label{thm:HH}
Let $(M, \xi_M)$ be a manifold with a co-orientable contact structure and $\Sigma\subset M$    a  co-oriented hypersurface. Then there exists a $C^0$--small isotopy sending $\Sigma$ to a convex
 hypersurface $\widetilde{\Sigma}$.
\end{theorem}

If $\dim M=2$ then Theorem \ref{thm:HH} holds -- according to a classical result of Emmanuel Giroux \cite{Gi} -- in a stronger form  with a $C^\infty$-small 
isotopy instead of a $C^0$-small isotopy.  The purpose of this article is to provide a   more accessible  proof of Theorem~\ref{thm:HH}. While the proof  follows  the overall strategy of  \cite{HH}  it is  significantly  different in  its implementation. In particular, we do not use any contact open  book techniques. Besides Theorem \ref{thm:HH} we do not discuss  in this paper any other results formulated in \cite{HH}.
 \medskip

 \medskip\noindent{\sl Acknowledgements.} The current paper was originated while the authors   visited the Institut for Advanced Study at Princeton, and was completed when the first author visited the Institut for Theoretical Study at the ETH, Zurich.
 The authors thank both Institutes for their hospitality. The paper grew up from a seminar at the IAS, where the second author   discussed Honda-Huang's work.
 The results of this paper were presented by the first author  at  seminars at Stanford, Neuch\^atel, Brussels and the ITS ETH, Zurich. The authors are very grateful to participants of all these seminars for   questions and comments. Our special thanks to Dietmar Salamon for asking many pointed questions and for his  proposed simplifications of several proofs presented in Section
 \ref{sec:conv-big} of this paper. In particular,   we use his proof for Lemma \ref{lm:W-convex}.
 We thank  Fran\c{c}ois-Simon Fauteux-Chapleau  for clarifying some of the aspects of Honda-Huang's paper, Ko Honda for answering many questions concerning their  proof, and Emmanuel Giroux for his comments on the preliminary version of this paper. We are grateful to N. Mishachev for making Fig. 
 \ref{fig:birthdeath}.


\section{Dynamics of vector fields}\label{sec:dynamics}
This  section   and  Section \ref{sec:conv-big} contains some background material which is mostly well-known.  \subsection{Lyapunov functions}
An isolated zero $p$ of a vector field $X$ on an $m$-dimensional manifold $\Sigma$  is called {\em non-degenerate} if $d_pX$ is non-degenerate, and it  is called an {\em embryo} or {\em death-birth} singularity if the corank of its linearization $d_pX$ is  equal to 1 and    the quadratic differential $d^2_pX:\Ker d_pX\to \Coker  d_pX$, which is  defined   up to scaling by a non-zero coefficient, does not vanish. We will call a non-degenerate   or   death-birth zero  {\em hyperbolic} if $d_pX$ has  no pure  imaginary (non-zero) eigenvalues.

Let $X$ be a vector field on a compact  manifold $\Sigma$.  Let  us endow $\Sigma$ with   a   Riemannian metric. A function $f:\Sigma\to\R$ is called   {\em Lyapunov} for $X$ if 
$ df(X)\geq C(||X||^2+||df||^2)$ for a positive constant $C$.
   Equivalently, one says that $X$ is a {\em gradient like} vector field for $f$.

It is a standard fact that  isolated hyperbolic zeroes, non-degenerate or embryos,     admit local Lyapunov function,   e.g. see \cite{Arn73}.  The {\em stable manifold}  of a hyperbolic zero is diffeomorphic to $\R^k$ for some $k=0,\dots, m$ in the non-degenerate case, and to $\R^k_+$ in the case of an embryo.   The dimension  $k$ of the stable manifold of a non-degenerate hyperbolic  zero  $O$ is called its  {\em index} and denoted by $\ind(O)$. For an  embryo the index  is usually defined  to be equal to  $k-\frac12$.

     \begin{figure}[ht]
 
\includegraphics[scale=0.85]{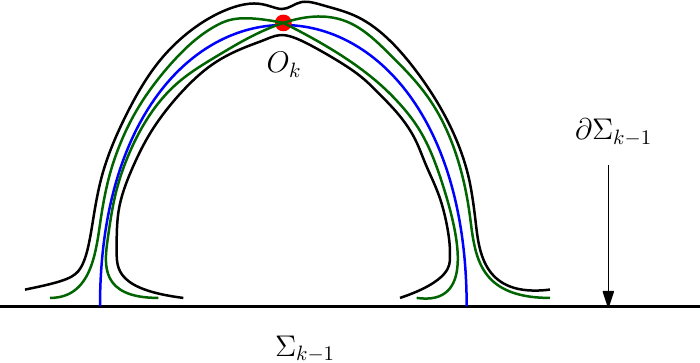}
 
\caption{Surrounding the stable disc of a zero $O_k$ by level sets of a Lyapunov function..}
\label{fig:handle_attachment}
\end{figure}

\begin{lemma}\label{lm:Lyapunov-criterion}

 Let $X$ be a vector field  with isolated hyperbolic zeroes which are non-degenerate or of embryo type on a closed $m$-dimensional manifold $\Sigma$.
Then $X$ admits a Lyapunov function if and only if the following conditions are satisfied:
\begin{itemize}
\item[(L1)]   every trajectory of $X$ originates and terminates at a  zero of $X$;
\item[(L2)]  there exists an  ordering  $O_1,\dots, O_N$ of zeroes such that there are no trajectories of $X$ which originate at $O_i$ and terminate in $O_j$ if $i>j$.
\end{itemize}
\end{lemma}

  \begin{proof} 
  If $X$ admits a Lyapunov function then both conditions (L1)  and (L2) are straightforward. Suppose that these conditions are satisfied.
We construct    a Lyapunov function  $f:\Sigma\to\R$ by extending it inductively to neighborhoods of stable manifolds of zeroes $O_j$.
 
We start with a local Lyapunov function near $O_1$ (which has to be of index $0$)  and    set     $f(O_1)=1$.   We assume  that $\Sigma_1:=\{f\leq \frac32\}$  is a small ball surrounding $O_1$, with boundary transverse to $X$.
 
Suppose that  we already constructed   $f$  on a domain $\Sigma_{k-1}: =\{f\leq k-\frac12\}$, $2\leq k\leq N$,   such that    zeroes $O_1,\dots, O_{k-1}$  and their stable manifolds  are contained in  $\Int\Sigma_{k-1}$ and $f(O_{k-1})=k-1$.  
The stable manifold   $P_k$  of $O_k$  transversely intersects $\p\Sigma_{k-1}$. Denote    $\wt P_k:=P_k\setminus\Int\Sigma_{k-1}$. Then  $\wt P_k$ is  an embedded disk (or  a half-disc, if $O_k$ is an embryo) of dimension $\ind(O_k)$ with boundary transverse to $\p\Sigma_{k-1}$. Extend     $f$    to a neighborhood $U_k\supset \Sigma_{k-1}\cup \wt P_k$  as a   Lyapunov function  for $X$ such that   $f(O_k)=k $  and the regular level set $\{f=k+\frac12\}$ is compact and  is contained in $U_k$, see  Figure~\ref{fig:handle_attachment} and \cite{CE12} for details. Denote $\Sigma_k = \{f\leq k+\frac12\}$. For $k=N$ we have $\Sigma_N= \{f\leq N \}=\Sigma$, and this completes the construction.  
\end{proof}
 
Note that condition (L2) is guaranteed by the  {\em Morse-Smale property}, i.e. transversality of stable and unstable manifolds for any pair of zeros. While  the Morse-Smale property can be arranged by a $C^\infty$-small perturbation, it is not clear to us whether this perturbation can be always done without destroying property (L1).

Lemma \ref{lm:Lyapunov-criterion} can be extended to 1-parametric families.
\begin{lemma}\label{lm:Lyapunov-criterion-param} Any family of vector 
fields $X_s$, $s\in[0,1]$, which satisfy conditions (L1) and (L2) admits a family  of    Lyapunov functions.
\end{lemma}
 \begin{proof} The space of Lyapunov functions for a given vector field $X$ is contractible,  because  a  convex linear combination of two Lyapunov functions for $X$ is again a Lyapunov function for $X$.   Also note that a Lyapunov  function $f_{s_0}$ for $X_{s_0}$  can  always be included into a family $f_s$ of   Lyapunov functions for $X_s$ for $s$ close to $s_0$.  Hence, the projection of the space of pairs  ((L1)+(L2) field,  Lyapunov function) to the space of (L1)+(L2)  fields is a micro-fibration with  a (non-empty!) contractible fiber, and hence, it is a Serre fibration, see \cite{Gro86, We05}.
 \end{proof} 

Let us also formulate a version of Lemma \ref{lm:Lyapunov-criterion} for a (trivial) cobordism. Let $W$ be an $(m-1)$-dimensional manifold with boundary and $\Sigma:=W\times [0,1]$. Denote by $y$ the coordinate  which corresponds to the second factor. Let $X$ be  a vector field on $\Sigma$ which coincides with $\frac{\p}{\p y}$ near $\p \Sigma$.
\begin{lemma}\label{lm:Lyapunov-criterion-rel}
Suppose that
\begin{itemize}
\item[(L$1'$)]   every trajectory of $X$ originates and terminates at a  zero of $X$ or at a point of $\p\Sigma$;
\item[(L2)]  there exists an  ordering  $O_1,\dots, O_N$ of zeroes such that there are no trajectories of $X$ which originate at $O_i$ and terminate in $O_j$ if $i>j$.
\end{itemize}
Then $X$ admits a Lyapunov function which is equal  to $y$ near $\p \Sigma$.
\end{lemma}
\begin{proof}
We construct $f$ by the  process described in the proof of Lemma
  \ref{lm:Lyapunov-criterion} with $W\times 0$ and $W\times 1$ playing the role of the first and last zeroes, $O_0$ and $O_{N+1}$. We then   adjust $f$ near $\p W\times[0,1]$,  by making it linear with respect to $y$ and  then scaling it to make equal to $1$ on $W\times 1$.
  \end{proof}

\subsection{Blocking collections}\label{sec:blocking} The material of this section is fairly standard and its various versions appear in many places (e.g. see \cite{GE71,Wi}). In particular, Lemma~\ref{lm:exist-discs} is a corollary of ~\cite[Lemma 2]{Wi}.

A non-vanishing vector field $X$ in a neighborhood of a hypersurface $V$ in an $m$-dimensional manifold $\Sigma$ is called  {\em in  general position}  with respect to $V$ if it has   Thom-Boardman-Morin  tangency singularities of type
$\Sigma^{1, \ldots ,1}$, see~\cite{Thom, Bo}. Let us fix a Riemannian metric on $\Sigma$. Arguing by induction over strata of tangency singularity, it is straightforward to prove the following statement (e.g. it is a corollary of Morin's normal forms \cite{Morin} for $\Sigma^{1, \ldots ,1}$-singularities).
\begin{lemma}\label{lm:Morin-eps} Suppose that $X$ is in  general position  with respect to $V$. Then there exists $\eps_0>0$ such that
 for any $\eps\in(0,\eps_0)$ there exists $\delta>0$ such that  any connected trajectory arc   of length $\eps$ contains a connected sub-arc of length $C(m)\eps$ which does not intersect   the $\delta$-neighborhood of $V$.  Here $C(m)$ denotes a   constant which depends only on the dimension $m.$  \end{lemma}
\begin{proof} We can assume that the vector field $X$ on a neighborhood of $V$ has a unit length. We will be measuring below   arcs  $\gamma$ of $X$-trajectories by the flow-parameter. This measurement, which we call {\em length}  is equivalent  to  the diameter of $\gamma$  for sufficiently short arcs.

For any point $p\in V$ and $\eps>0$ denote $ \gamma_\eps(p)=\bigcup\limits_{u\in[-\eps,\eps]}X^u(p)$.
There exists $\eps_0>0$ such that for each point $p\in  V$ the arc $\gamma_{3\eps_0}(p)$  intersects $V$ at no more than $m$ points, and moreover $V\cap  \gamma_{3\eps_0}(p)\subset  \gamma_{ \eps_0 }(p).$ Given $\delta>0$ denote
by $N_\delta(V)$ the $\delta$-tubular neighborhood of $V$. 
For any $\eps<\eps_0$ there exists $\delta>0$ such that for every  $p\in V$  the intersection $\gamma_{3\eps_0}(p)\cap
N_\delta(V)$ consists of no more than $m$ components  of length  $<\frac{\eps}{2m+2} $.
Any trajectory arc $\sigma$ of length $\eps$ which intersects $N_\delta(V)$ is contained in $\gamma_{3\eps_0}(p)$ for some $p\in V$. Hence $ N_\delta(V)\cap \sigma$ consists of no more than $m$  arcs of length $<\frac{\eps}{2m+2} $. Thus, the complement $\sigma\setminus N_\delta(V)$  contains an arc of length $>\frac{\eps}{m+1}-\frac\eps{2m+2}=
\frac\eps{2m+2}.$
\end{proof}
 
\medskip

Let $X$ be a vector field on a compact $m$-dimensional manifold $\Sigma$, possibly with boundary. 
  Given $\eps>0$,   
  a  finite collection  $\{D_j\}_{1\leq j\leq K}$ of transverse to $X$ embedded into $\Int \Sigma$  codimension one discs  of diameter $<\eps$ is called {\em $\eps$-blocking  }
     if any connected  trajectory arc   of diameter $>\eps$ intersects $\bigcup \limits_1^K \Int D_j$.  
     
\begin{lemma}\label{lm:exist-discs} Let $X$ be a vector field on a compact $m$-dimensional manifold $\Sigma$, possibly  with boundary. Suppose that all zeros of $X$ are  in $\Int \Sigma$, isolated and hyperbolic (non-degenerate or embryos).    Suppose that $X$ is in  general position with respect to $\p\Sigma$.
Then for any $\eps>0$  the field $X$  admits an $\eps$-blocking collection.
\end{lemma}
\begin{proof}

\noindent {\sl Part I.} Suppose first that the vector field $X|_{\Sigma}$ admits a Lyapunov function $f:\wt\Sigma\to\R$. 
without critical points.  
Suppose that $\min f=0,\max f=1$ and choose $N$  large enough to guarantee that any connected arc of an  $X$-trajectory in $\{\frac jN\leq f\leq \frac{j+1}N\}$ has its diameter $<\frac\eps2$,   $j=0,\dots, N-1$. Suppose that $\eps$ is chosen $<\eps_0$ from Lemma \ref{lm:Morin-eps} and $\delta$ is chosen so small that  any connected trajectory arc   of length $\eps$ contains a connected sub-arc of length $\frac\eps{C(m)}$ which does not intersect   the $\delta$-neighborhood of $V$.  Here $C(m)$  is the constant from Lemma \ref{lm:Morin-eps}.
  Choose an interior tubular collar $\p\Sigma \times[-1,0]\subset\Sigma$ such that $\p\Sigma=\p\Sigma\times 0$ and $\p\Sigma\times (-1)$ is at a distance $\delta$ from $\p\Sigma$. Denote $ \Sigma_0:= \Sigma\setminus\left(\p \Sigma \times(-1,0]\right)$.
For each $j=1,\dots, N-1$ choose   finitely many closed discs of radius $\eps$ in  $\Int \{f=\frac jN\}$ whose interiors cover  $\{f=\frac jN\}\cap \Sigma_0$.   By shifting these discs to disjoint level sets  $\{f=t_{j,k}\},\; t_{j,k}\in(\frac{2j-1}{2N},\frac{2j+1}{2N}),$ we get  the required $\eps$-blocking collection.

\smallskip\noindent{\sl Part II.}
In the general case  let us choose any smooth function $f:\Sigma\to\R$.
 Let us surround zeroes of $X$ by  the union $B$ of disjoint  closed  $\eps$-balls. We can assume that $X$ is in  general position with respect to $\p B$. Denote $\wt\Sigma:=\Sigma\setminus \Int B$.

 Denote $$\wt\Sigma_+:=\{df(X)\geq 0\},\; \wt\Sigma_-:=\{df(X)\leq 0\},\;   V:=\{df(X)= 0\}=\wt\Sigma_+\cap\wt\Sigma_-.$$
By  $C^\infty$-perturbing   $f$, if necessary, we can arrange that $V$ is a codimension 1 submanifold, and $X$ is   in  general position   with respect to $V$. Let us assume that $\delta>0$ is chosen in such a way that any connected trajectory arc   of length $\eps$ contains a connected sub-arc of length $\frac\eps{C(m)}$ which does not intersect   the $\delta$-neighborhood of $V$.  
 Consider a $\delta$-tubular neighborhood $N\supset V, N\subset\Int\Sigma$. Denote $\wh\Sigma_\pm:=\wt\Sigma_\pm\setminus\Int N$. We can assume that $\p \wh\Sigma_\pm$ is in general position with respect to $X$. By applying Part 1 we can construct  $\frac{\eps}{2C(m)}$-blocking collections for   $\wh\Sigma_+$ and $\wh\Sigma_-$.
The union of these collections is the required $\eps$-blocking collection 
for $X$ on  $\Sigma$.

\end{proof}

   Note that a compact arc $\gamma$ of a non-constant trajectory $X$ has a {\em flow-box} neighborhood  $U=D\times [0,c]$ such that $D\times 0$ is an embedded transverse disc, and $x\times[0,c], x\in D$ are trajectories of $X$. Denote by $\pi_U:U\to D$ the projection of the flow-box neighborhood to the first factor.
   We will call an  $\eps$-blocking collection $\{D_j\}$ generic, if for any flow-box $U$  projections  $ \pi_j := \pi_U|_{ \p D_j\cap U}:\p D_j\cap U\to D$ are transverse to each other.  Any $\eps$-blocking collection can be made generic by a $C^\infty$-perturbation.

   \subsection{Plugs}\label{sec:plug}

Given an $\eps$-blocking collection  $  \{D_j\}$,  let us thicken  discs $D_j$ to disjoint   flow-boxes  $Q_j=D_j\times[0,a]$ such that    intervals $x\times[0,a]$, $x\in D_j $  are  time $a$  trajectories of $X$ originated at $x\in D_j=D_j\times 0$. We will assume that $a$ is chosen small enough to guarantee that flow-boxes $Q_j$ have diameter $<2\eps$. 

Let $D$ be an $(m-1)$-dimensional disc. A vector field $Y$  on $D\times[0,a]$ is called 
a  {\em $\sigma$-plug}  if the following conditions are satisfied:
\begin{itemize}
\item[P1.] $Y$ coincides with $\frac{\p}{\p y}$ on $ \p Q$, where $y$ is the coordinate on $D\times[0,a]$ corresponding to the second factor;
\item[P2.] $Y$ satisfies the Morse-Smale condition and  admits a Morse Lyapunov function;  
\item[P3.] for  any point $p\in D$   with $\dist (p,\p D)>\sigma$ the  trajectory of $Y$    through     $\ p\times 0$ converges  to a critical point of $Y$;\item[P4.] given any  point $p\in D$  the  trajectory of $Y$    through    $p\times 0 $  either converges  to a critical point of $Y$, or exit $Q$ at a point $ p'\times a   $ where $\dist (   p', p)<\sigma$.
\end{itemize}
\begin{lemma}\label{lm:sigma-plug}
Let   $\Sigma$ be a closed  manifold of dimension $m$, and    $X$    a  vector field on $\Sigma$ with non-degenerate hyperbolic zeroes.  
Let $\{D_j\}$ be  a  generic  $\eps$-blocking collection,  and 
    $\{Q_j\} $   a collection of their    disjoint flow-boxes  of diameter $ <2\eps$. Then, if $\eps$ is sufficiently small  there exists    $\sigma>0$ such that  by replacing      for each $j$ the vector field $X|_{Q_j}$ by a    $\sigma$-plug $Y$ one gets  a vector field $\wh X$ which satisfies condition (L1) and such that
 all  its trajectories   has diameter $<10\eps$. Moreover,  property (L1) survives a sufficiently small $C^1$-perturbation of  $X$ away from flow-boxes and    neighborhoods of zeroes of $X$. 
 \end{lemma}

 \begin{proof} 
Let us surround zeroes of $X$ by disjoint  balls $B_1, \dots, B_K$  of radius $<\eps$ such that the flowboxes $Q_j$ are contained in $\Sigma':=\Sigma\setminus\bigcup\limits_1^K B_j$. Any connected arc $\gamma$ of diameter $5\eps$ of an $X$-trajectory has a connected sub-arc $\gamma'\subset \gamma$ of length $2\eps$ in  $\Sigma':=\Sigma\setminus\bigcup\limits_1^K B_j$, and hence by definition of an $\eps$-blocking system there exists a sub-arc $\gamma''\subset\gamma$ of length $<\eps$ with  both its ends at
   $\Int D_i$ and $\Int D_j$ for some  $i,j$.
    By compactness argument we can find smaller closed discs $D_{j,-}\subset \Int D_j$ such that in  the above condition one can replace   $\Int D_j$  by  $D_{j,-}$.
 
 The genericity  property for the blocking collection $D_j$ implies that if  an  arc $\gamma$ of an $X$-trajectory does not intersect $\bigcup\limits_1^K \Int D_{j}$ then it cannot have more than $m-1$ intersection points with $\p D_j$. Hence, we can choose the discs    $D_{j,-}$ so close to $D_j$ that any arc    which does not intersect
 $\bigcup\limits_1^K D_{i,-}$ intersects no more than $m-1$ annuli $A_j:=D_j\setminus\Int D_{j,-}$. 
 for any arc $\delta$    connecting points $p\in D_i$ and $q\in D_j$ let $A_\delta:T_pD_i\to T_q D_j$ be the linearized monodromy map along the flow of $X$. By compactness of $\Sigma'$  there is a universal upper bound $C$ for $||A_\delta||$ for  all arcs $\delta\subset\Sigma'$ of diameter $\leq\eps$ connecting points on blocking discs.
  Choose $\sigma<\frac1{Cm+1}\min\limits_j(\dist(\p D_{j,-}.\p D_j)$. We claim that the vector field $\wh X$ obtained by replacing $X|_{Q_j}$ by $\sigma$-plugs has the required properties. Indeed,  consider any trajectory  $\wh\gamma$ of $\wh X$ which enters a flow-box $Q_i$ through  a point $p\in  D_{i,-}=D_{i,-}\times 0$. Let $\gamma$ be an $X$-trajectory through $p$ which intersect $D_{j,-}$ at a point $p'$, and  does not contain  any other points from  $\bigcup\limits_1^K D_{k,-}$. Then if $\wh\gamma$ is not blocked in $Q_i$, or any other of $<m$ plugs $Q_{k}$ for discs $D_{k}$ which intersect  $\gamma$, then it enters $Q_j$ through a point $p''$ with $\dist(p',p'')< Cm\sigma$, and thus $\dist(p'',\p D_i)>\sigma$. But this   means that the trajectory $\wh\gamma$ converges to a   zero of $\wh X$ in $Q_j$.  Moreover,  the trajectories of  the vector  field $\wh X$  have their diameter bounded by $10\eps$. Finally, we observe that the above analysis remains valid if   $X$ is  perturbed  by a sufficiently $C^1$-small homotopy outside flow-boxes.

   \end{proof}

\begin{figure}[ht]
 
 \includegraphics[scale=0.7]{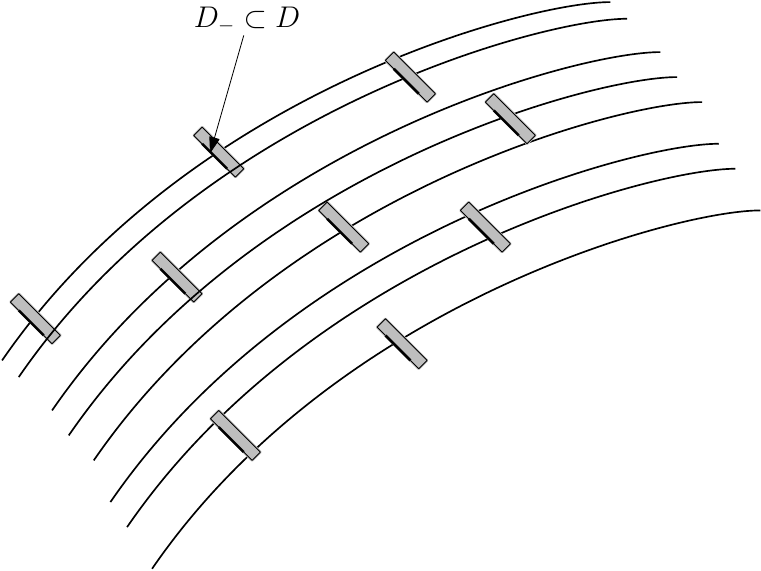}
 
\caption{ Blocking discs   with  their flow-boxes. }
\label{fig:blocking_collection}
\end{figure}

\section{Contact convexity}\label{sec:conv-big}
\subsection{Characteristic foliation}
 
Let $M$ be a contact manifold of  dimension $2n+1$ with a co-oriented   contact structure $\xi=\{\alpha=0\}$.  The volume form $\mu:=\alpha\wedge(d\alpha)^n$ defines an orientation of $M$, and $(d\alpha)^n|_{\xi}$  defines an orientation of $\xi$. If $n=2k+1$ then the former orientation, and if  $n=2k$ than the latter orientation,  depends only on $\xi$.

Let $\Sigma\subset M$ be a co-oriented hypersurface.  If  $\nu$ is a vector field defining its  co-orientation then the orientation of  $\Sigma$ is given  by the   $2n$-form  $\iota(\nu)\mu$.     At any point $p\in\Sigma$ where $\xi_p\pitchfork T_p\Sigma ,$ there  is defined a {\em characteristic line} $\ell_p:=\Ker \{d\alpha|_{\xi_p\cap T_p\Sigma}\}\subset \xi_p\cap T_p\Sigma$.  Note that the co-orientation of $\Sigma$ defines a co-orientation of  $\xi_p\cap T_p\Sigma$ in $\xi_p$. We orient $\ell_p$ by a vector $X_p\in\ell_p$ such that the 1-form  $\iota(X_p)d\alpha|_{\xi_p}$ defines   that  co-orientation.  

The line field $\ell$,  which is  defined in the complement of the tangency locus $T$  between $\xi$ and $\Sigma$, integrates to a singular foliation on $\Sigma$ with singularities at the points of  $T$. We will keep the notation $\ell$ for this foliation, and write $\ell_\Sigma$, $\ell_\xi$ or $\ell_{\xi,\Sigma}$ when it is  important to stress the dependence of $\ell$ on $\Sigma, \xi$, or both.

The singular locus $T$  splits   as a union of disjoint closed subsets, $T=T_+\cup T_-$, where $T_+$ (resp. $T_-$) consists of {\em positive} (resp. {\em negative}) points, where  the orientations  of $\xi_p$ and $T_p(\Sigma)$  coincide (resp. opposite). On    neighborhoods $U_\pm\supset T_\pm$, $U_\pm\subset\Sigma$, the form $d\beta$, $\beta=\alpha|_\Sigma$, is symplectic. We define a vector field $X$ on $\Sigma$ which directs  $\ell$ as equal to the Liouville field $d\beta$-dual to $ \beta$ on $U_+$,  and as a vector field     $d\beta$-dual to $- \beta$ on $U_-$, and extend it to the rest of $\Sigma$ as any non-vanishing vector field. The following lemma is due to E. Giroux in \cite{Gi} and was  pointed out to us by  D. Salamon. It provides an equivalent characterization of a vector field directing the characteristic foliation.
 Choose a  positive volume form $\rho$ on $\Sigma$  equal to $ (d\beta)^n$  on $U_+$  and  to $ -(d\beta)^n$    on $U_-$. 
 \begin{lemma} 
\label{lm:Dietmar2}
The vector field $X$ defined by the equation
  \begin{align}\label{eq:char}
  \iota(X)\rho=n\beta\wedge (d\beta)^{n-1}.
  \end{align}
    directs the characteristic foliation $\ell$.
    \end{lemma}
    \begin{proof}
 On $U_\pm$  equation \eqref{eq:char} is equivalent to
  $\iota(X)d\beta=\pm \beta$, i.e. $X$ coincides with the Liouville field  dual to
  $\beta$ on $U_+$  and to the $d\beta$-dual to $- \beta$ vector field  on $U_-$.
   Elsewhere, $X\neq 0$ and  $  \iota(X)(\beta\wedge(d\beta)^{n-1})=0.$ If $n=1$ this implies that $\beta(X)=0$ and hence, $X\in\ell$. For $n\geq 2$ we have  
  \begin{align*} &0=\iota(X)(\beta\wedge(d\beta)^{n-1})=
  \beta(X)(d\beta)^{n-1}+(n-1) \left(\iota(X)d\beta\right)\wedge\beta\wedge  d\beta^{n-2}. 
  \end{align*}
    By restricting the from on the right hand side of the above equation  to $\Ker\beta$ we conclude that   $\beta(X)=0.$ This is because when restricted 
    to  $\Ker\beta$   the form $\left(\iota(X)d\beta\right)\wedge\beta\wedge d\beta^{n-2}=0.$  But then we get,
    $ \left(\iota(X)d\beta\right)|_{\Ker\beta} =0$. Indeed, the form $d\beta$ descends   as a symplectic form to the $(2n-2)$-dimensional quotient space $Q_p:=(\xi_p\cap T_p\Sigma)/T\ell_p$ as a symplectic form. Hence, the multiplication  by   $d\beta^{n-2}$ defines an isomorphism between $1$- and $(2n-3)$-forms on $Q_p$, and the claim follows.

\end{proof}
 Let us recall that the contact structure on a neighborhood of a hypersurface $\Sigma$ is determined by its restriction to the hypersurface.
  \begin{prop}[A. Givental, \cite{AG}]\label{prop:AG}
 Let $\xi=\Ker\alpha, \xi'=\Ker\alpha'$ be two contact structures defined on a neighborhood of $\Sigma=\Sigma\times 0\subset\Sigma\times\R$. Suppose that 
 $\alpha|_{\Sigma}=h \alpha'|_{\Sigma'}$ for a positive function $h:\Sigma\to\R$.
 Then there exists a diffeomorphism $g:\Op\Sigma\to\Op\Sigma$ which  is fixed on $\Sigma$ and such that $dg(\xi)=\xi'$.
 \end{prop}
In fact, the statement formulated in \cite{AG} is slightly weaker. We thank D. Salamon for  providing the details of the proof of the above result.

All singularities of  a vector field $X$  directing the characteristic foliation $\ell$ can be made non-degenerate  and hyperbolic  by a $C^\infty$-small perturbation of $\Sigma$, see e.g. \cite{CE12}.   For a generic 1-parametric family of characteristic foliations, the directing vector field $X_s$ can  also  have (hyperbolic) embryo singularities for isolated values of the parameter $s$.
 
 We will need the following version of Lemma \ref{lm:exist-discs} for a vector field  $X$ directing a characteristic foliation.  Define the {\em standard contact $(2n-1)$-disc} $(D=D^{2n-1},\xi_\st)$ as contactomorphic to the  hemisphere
$D=S^{2n-1}_+:=S^{2n-1}\cap \{y_n\geq 0\}$ endowed with the contact structure $\xi_\st:=\left\{\sum\limits_1^n( x_jdy_j-y_jdx_j)|_{S^{2n-1}}=0\right\}$.

\begin{lemma}\label{lm:exist-discs-cont} Let $X$ be a vector field directing a characteristic foliation on a closed hypersurface $\Sigma$ in a contact manifold of dimension $2n+1$. Suppose that all zeros of $X$ are  isolated and hyperbolic (non-degenerate or embryos).   Then for any $\eps>0$  the field $X$  admits an $\eps$-blocking collection $\{D_j\}$  which consists of standard contact $(2n-1)$-dimensional  discs.
\end{lemma}
\begin{proof} We only need to ensure  that discs forming the blocking collections can be chosen contactomorphic to the standard  contact disc. We recall that in the proof of   Lemma \ref{lm:exist-discs} discs $D_j$  arise as  elements of a covering of a transverse contact hypersurface.  But the covering can  always be  chosen  to be formed   by standard small Darboux balls.
\end{proof}

\subsection{Lyapunov functions for characteristic foliations}

For a  vector field $X$ directing a characteristic foliation   $\ell$ on a hypersurface $\Sigma$ stable manifolds of  positive zeroes   are isotropic with respect to $d\beta$,  while unstable are coisotropic, see \cite{CE12}.
Near  negative zeroes the field $-X$ is Liouville, and thus stable manifolds of negative zeroes    are  co-isotropic  while unstable are isotropic,    In particular,  a  local Lyapunov function on $\Op (T_+\cup T_-)$, have critical points of index $\leq n$ at  the positive points, and of  index $\geq n$ in the negative ones. It is important to note  that  critical points   of index $n$ can be either negative or positive.
The stable manifold of a positive (resp. negative)  embryo   is an isotropic
(resp. co-isotropic) half-space.

We call a Lyapunov function $f:\Sigma\to\R$ for $X$ {\em good}
if there exists a regular value $c$ such that all  positive zeroes of $X$ are in $\{f>c\}$ and all negative ones are in $\{f<c\}$.   Sometimes we will call $f$  a Lyapunov function for  $\ell$, rather than  $X$.

 Following Giroux, we call a trajectory $\gamma$ of $X$ a {\em retrograde connection} if it originates at a negative point of $X$ and terminates at a positive one.
 Lemma \ref{lm:Lyapunov-criterion} implies  
 \begin{corollary}\label{cor:good-Lyapunov}
 Suppose that a vector field $X$ satisfies conditions (L1) and (L2) from   Lemma \ref{lm:Lyapunov-criterion}.  Then it admits a good Lyapunov function if and only if it  has no retrograde connections. In particular, any  $X$ which satisfies (L1) and the   Morse-Smale condition admits a good Lyapunov function. \end{corollary}
\begin{proof} If there are no retrograde connections, then one can always order zeroes in such a way that positive zeroes go first, and hence the construction in \ref{lm:Lyapunov-criterion} yields a good Lyapunov function. The necessity of the absence of retrograde connections for existence of good Lyapunov function is straightforward.
\end{proof}

Similarly,  using Lemma \ref{lm:Lyapunov-criterion-param} we get a parametric version of this statement.
  \begin{corollary}\label{cor:Lyapunov-convexity-param} Any family  $X_s$, $s\in[0,1]$,  which satisfy (L1) and (L2)  and have no retrograde connections admits a family  of  good Lyapunov functions.
\end{corollary}
 \begin{proof} An additional observation which is needed for the proof, in addition  to the argument in \ref{lm:Lyapunov-criterion-param}, is  that the space of good Lyapunov functions is contractible.  Indeed,   if we normalize  Lyapunov functions by the condition that the  $0$ level  is separating positive and negative points, then their convex linear combination   is again a good Lyapunov function.
 \end{proof}

\subsection{Flavors of contact convexity}\label{sec:conv}
The notion of contact convexity was first defined in \cite{EG91}, and then explored by Emmanuel Giroux,  see \cite{Gi}, Ko Honda, see \cite{Ho00}, and others.
\begin{definition}
\begin{enumerate}\item  A hypersurface $\Sigma\subset (M,\xi)$ is called {\em convex} if it admits a transverse contact vector field $\Upsilon$.
\item  A hypersurface $\Sigma$ is called {\em Weinstein convex} if
    its characteristic foliation $\ell$ admits a good Lyapunov function.
  \end{enumerate}
  \end{definition}
  As we will see below in Lemma \ref{lm:W-convex}, Weinstein convexity is a stronger condition which  implies   convexity.

  E. Giroux proved in \cite{Gi} that for  $2$-dimensional surfaces contact convexity can be achieved by a $C^\infty$-perturbation.
    
    Using Corollary \ref{cor:good-Lyapunov} we can equivalently characterize Weinstein convexity by conditions (L1) and (L2)  (or, equivalently, existence of  {\em any} Lyapunov function) for $X$ and absence of retrograde connections.
  As it was pointed out above, condition (L2) is implied by the    Morse-Smale condition, which is generic for individual hypersurfaces.   
 \bigskip

Following Giroux, the set   $S:=\{x\in\Sigma;\Upsilon(x)\in\xi_x\}$ is called the {\em dividing set} of $\Sigma$.
\begin{lemma}[E. Giroux, \cite{Gi}]\label{lm:div-form}
Suppose $X$ is  a contact vector field transverse to a hypersurface $\Sigma$ and $S$ the corresponding dividing set. Let $t$   be the flow   coordinate such that $\Sigma=\{t=0\}$ and $X=\frac{\p}{\p t}$. Then  $\xi$ on $\Op  \Sigma$ can be defined by   a  contact 1-form
$f(x)dt+\beta$, where  $f:\Sigma \to\R$ is a function transversely changing sign across $S$. \end{lemma}
Note that the contact condition implies that $df\neq 0$ along $S$, and $\alpha|_{S}$ is a contact form.  In particular, the characteristic foliation $\ell_\Sigma$ transverse to $S$.

  Hence, we have the following:
\begin{lemma}[E. Giroux, \cite{Gi}]\label{lm:div-set}
 Dividing  set $S$ is a smooth submanifold, which is transverse to the characteristic foliation, and independent of the choice of a contact vector field transverse to $\Sigma$,  up to an isotopy  transverse to the characteristic foliation.
\end{lemma}
Indeed, the space of contact vector fields transverse to $\Sigma$ is convex subset of the vector space of all contact vector fields, and hence, contractible.

\medskip The  dividing hypersurface  $S\subset\Sigma$ divides $\Sigma$ into
$\Sigma_+:=\{f>0\}$, $\Sigma_-:=\{f<0\}$.
The form $\alpha=(f(x)dt+\beta)|_{\Sigma\setminus S}$ can be divided 
by $f$,
$$\frac{\alpha}{f}=dt+\frac{\beta}f.$$
Denote $\lambda_\pm:=\frac{\beta}f|_{\Sigma_\pm}$.
The contact condition then is equivalent to $(d\lambda_\pm)^n\neq 0$.
In other words,  $\lambda_\pm$ are Liouville forms on $\Sigma_\pm$.
Note that the corresponding Liouville fields $Z_\pm$ directs the characteristic foliation on $\Sigma$. Indeed, $\lambda_\pm\wedge \iota(Z_\pm) d\lambda_\pm=\lambda_\pm\wedge\lambda_\pm=0.$

\begin{lemma}\label{lm:Giroux-inverse}
Let $\Sigma\subset (M,\xi=\Ker\alpha)$ be   a co-oriented hypersurface.  Denote $\beta:=\alpha|_\Sigma$.
Then $\Sigma$ is convex if and only if there exists a function $f:\Sigma\to\R$ such that the form $\beta+fdt$ is contact on $\Sigma\times\R$.
\end{lemma}
\begin{proof} The necessity is a reformulation of Lemma \ref{lm:div-form}. To see the sufficiency we observe that $\Sigma=\Sigma\times 0\subset (\Sigma\times\R,\Ker(\beta+fdt))$ is convex because the field $\Upsilon:=\frac{\p}{\p t}$ is  manifestly contact.
On the other hand, by Proposition \ref{prop:AG} neighborhoods of $\Sigma$ in $(M,\xi)$ and  $\Sigma\times 0$ in  $(\Sigma\times\R,\Ker(\beta+fdt))$ are contactomorphic.
\end{proof}

 \medskip

  \begin{lemma}\label{lm:W-convex} Any Weinstein convex hypersurface is convex.
  \end{lemma} 
   \begin{proof} 
   According to Lemma \ref{lm:Giroux-inverse} it is sufficient to find a function $f:\Sigma\to\R$ such that
   the form $\wt\alpha:=\beta+fdt$ is contact.
 We claim that in turn  this   condition is equivalent to the inequality
  \begin{align}\label{eq:Salamon}
  f(d\beta)^n+n\beta\wedge (d\beta)^{n-1}\wedge df>0.
 \end{align}
Indeed,
\begin{align*}
&\wt\alpha\wedge (d\wt\alpha)^n=(\beta+fdt)\wedge
(d\beta+df\wedge dt)^n=fdt\wedge(d\beta)^n+ n\beta\wedge (d\beta)^{n-1}\wedge df\wedge dt\\
&=
dt\wedge\left( f(d\beta)^n+n\beta\wedge (d\beta)^{n-1}\wedge df\right),
\end{align*}
and hence, the inequality $ \wt\alpha\wedge (d\wt\alpha)^n >0$ is equivalent to \eqref{eq:Salamon}.
  
   Suppose $\rho$ is a volume form on $\Sigma$ chosen as in Lemma \ref{lm:Dietmar2}, and the vector field $X$ directing $\ell$ satisfies equation
   \eqref{eq:char}.  Let $h:\Sigma\to\R$ be a good Lyapunov function on $\Sigma$.
   We will assume that $S:=\{h=0\}$ is a regular level set of $h$ separating  values in negative and positive zeroes of $X$. In particular, $h|_{U_+}<0, h|_{U_-}>0$ for neighborhoods $U_{\pm}\supset T_\pm$ of the singular point loci.
   
   Define a function $g:\Sigma\to\R$ by the equation
   $ (d\beta)^n=g\rho$. 
   We have $g|_{U_+}>0$ and $g|_{U_-}<0$. Hence, $gh<0$ on $U:=U_+\cup U_-$. We have $dh(X)>0$  on $ \Sigma\setminus ((T:=T_+\cup T_-)\cup S)$. Furthermore,   
   $dh(X)-hg>0$  on     a neighborhood  of  $S=\{h=0\}$, and   for a sufficiently large constant $C>0$ we have $$dh(X)-hg+Ch^2dh(X)>0$$ everywhere on $\Sigma$. 
   The 
 function $f:\Sigma\to\R$,  which satisfies \eqref{eq:Salamon}  can now be defined by the formula $f:=-e^{\frac{Ch^2}2}h.$ 
   Indeed,   we have $df=-e^{\frac{Ch^2}2}(1+Ch^2)dh$ and therefore,
   
   \begin{align*}
   &fd\beta^n+n\beta\wedge (d\beta)^{n-1}\wedge df=fg\rho+n(\iota(X)\rho)\wedge df.
   \end{align*}
   But for any 1-form $\gamma$ we have  $\iota(X)\rho\wedge\gamma=-\iota(X)(\rho\wedge\gamma)-\gamma(X)\rho=\gamma(X)\rho,$ and hence
      \begin{align*} 
& fg\rho+n(\iota(X)\rho)\wedge df=(fg-df(X))\rho=-e^{\frac{Ch^2}2}\left(hg-(1+Ch^2)dh(X)\right)\rho\\
&=e^{\frac{Ch^2}2}(dh(X)-hg+Ch^2dh(X))>0.
\end{align*}
  \end{proof}
Corollary  \ref{cor:good-Lyapunov} and 
Lemma  \ref{lm:W-convex}  reduce  Theorem \ref{thm:HH} (in a stronger form replacing convexity by Weinstein convexity)  to
\begin{theorem}[Honda-Huang,\cite{HH}]\label{thm:main-precise}
Let $\Sigma\subset (M,\xi)$ be a co-orientable hypersurface in a contact manifold with a co-orientable contact structure. Then there exists a $C^0$-small isotopy deforming $\Sigma $ into $\wt\Sigma\subset (M,\xi)$ such that the characteristic foliation induced on $\wt \Sigma$  satisfies condition  (L1) and the Morse-Smale property.   
\end{theorem}
For the case  $\dim M=2$   this result can be deduced from    E. Giroux's theorem about $C^\infty$-genericity of contact convexity in $3$-dimensional contact manifolds, \cite{Gi}.
 
  \section{Construction of plugs}\label{sec:plugs}
     \subsection{Main proposition}

In view  of Lemmas \ref{lm:sigma-plug}  and \ref{lm:exist-discs-cont}, the proof of Theorem \ref{thm:main-precise} will be completed if  for any $\sigma>0$ one    can  create a  $\sigma$-plug by a $C^0$-small isotopy of the flow box of a  standard, transverse to the flow contact disc. The next proposition asserts that this is possible. As the statement will be proven by induction, we need   $\sigma$-plugs with some additional properties  in order for the induction to go through.
We discuss this in Proposition~\ref{prop:main-plug}.

Let $(D^{2n-1},\alpha_{st})$ be the standard contact disc. Choose $c,b>0$ and 
consider $U^b:=D\times T^*[0,b]$ endowed with a contact form $\alpha_\st+xdy$.
Denote $$Q^b:=\{x=0\}\subset  U=D\times[0,b],\; U^b_c:=\{|x|\leq c\}\subset U^b.$$
We will omit the superscript $b$ when $b=1$.
\begin{prop}\label{prop:main-plug}
  For any  positive $ \eps$ and $\sigma\ll\eps$ there exists an    isotopy
 $ h_{s}: Q\to  U_\eps$, $s\in [0,1]$, which is fixed on $\Op \p  Q$, begins with       the inclusion $h_0: Q\hookrightarrow U_\eps$ and has the following properties:
 \begin{itemize}
  \item[a)] $(Q,X_{1})$ is a $\sigma$-plug, where we denoted by $X_{s}$ the vector field directing the characteristic foliation $\ell_{s}$    induced by $h_{s}^*(\alpha_\st+xdy)$; 
 \item[b)] for any $\sigma_1\ll\sigma$ there exists a family of compact  manifolds with boundary $C^+_s\subset\Int  Q$ and $C^-_s\subset\Int  Q,$ and an extension of  the isotopy $h_s$ to a 2-parametric isotopy
 $h_{s,t}, \; s,t\in[0,1]$,  such that
 \begin{itemize}
 \item[(i)] $h_{s,0}=h_s,\; h_{0,t}=h_0$ for all $s,t \in[0,1]$;
  \item the foliation $\ell_{s,1}$  induced on $h_{s,1}( Q)$ has no singular points;
  \item[(ii)]  for any fixed $s\in[0,1]$ the isotopy $h_{s,t}$, $t\in[0,1]$, is $\sigma_1$-small in the $C^0$-sense and supported in a $\sigma_1$-neighborhood of $C_s$;
  \item[(iii)] For each $s\in[0,1]$ the submanifold $C^+_s$ (resp. $C_s^-$)  contains all positive (resp. negative) singularities of $X_s,$ $C_s^+ \cap 
  C_S^- = \emptyset$ and $C^+_s$ (resp. $C^-_s$) is  invariant with respect  to the backward (resp. forward) flow of $X_s$;
 
 \item[(iv)]  there exists a family of  generalized Morse  Lyapunov functions $\psi_{s,t}:Q\to\R$ for $X_{s,t}$ such that  $\psi_{s,t}|_{\Op\p  Q}=y$;
 \item[(v)] there exists a stratified  $(n-1)$-dimensional subset $E\subset  Q\cap\{y=0\}$ which contains all  the intersection points of $Q\cap\{y=0\}$ with stable manifolds of   positive singular point of $X_{s,t}$ for all $s,t\in[0,1]$.  
  \end{itemize}
  \end{itemize}
   \end{prop}

 We will refer to the statement of Proposition \ref{prop:main-plug} as an {\em installation of a $\sigma$-plug  of height $\eps$   over $Q=D^{2n-1}\times [0,1]$}, where $(D^{2n-1},\alpha_\st)$ is  the standard  contact disc. The same statement with $Q, U_c$ are replaced by $Q^b$ and $U^b_c$   will be referred as an installation over $Q^b$. 
  Note that the contactomorphism
 $ U^b_a\to U^1_{ab}$ induced by the   linear map $(x,y)\mapsto(b x,\frac yb)$ of the second factor  always allows us to reduce the installation to the case   $b=1$.
 
  If we further replace $(D,\alpha_\st)$ in  the statement  of \ref{prop:main-plug}  by any compact  $(2n-1)$-dimensional manifold $V$ manifold  with boundary (and possibly with corners) and with a fixed contact form $\alpha$, we will  say that we are {\em installing  a $\sigma$-plug   of height $\eps$   over $V\times[0,b]$.} 
 \subsection{Plan of the proof of Proposition \ref{prop:main-plug}}
 We begin the proof in Section \ref{sec:plug-height} by showing that  Proposition \ref{prop:main-plug} can be deduced from a weaker  
Proposition \ref{prop:main-plug-weak}, where the required isotopy $h_s$ is constructed in $U_K$ for a large  $K$ which may depend on $\sigma$,    rather than $U_\eps$ for an arbitrary small $\eps$.  This is done by a scaling argument.  One  of the subtleties here is  that  contact scalings  are better adjusted  to  Carnot-Caratheodory type metrics, rather than  Riemannian ones. Thus, we have to analyze separately an effect of the scaling on measurements in   directions tangent and transverse to contact planes.

The continuation of the proof is by induction on dimension $2n-1$.  
   Lemma \ref{lm:2d-plug}  in Section \ref{sec:2-plug}   serves as the base of the induction for $n=2$, as well as an   important ingredient in the proof of  the induction  step.  The Giroux-Fuchs creation-elimination construction, which we recall in Section \ref{sec:GF}, is an essential ingredient to the proof of Lemma \ref{lm:2d-plug}.
  
By taking a product of the two-dimensional  plug constructed in Lemma \ref{lm:2d-plug}  with a $(2n-2)$-dimensional Weinstein domain $(W,\lambda)$  and appropriately adjusting the product over  $\Op \p W$ we construct in Section \ref{sec:prelim-plug}  a $2n$-dimensional {\em preliminary plug} over  $((W\times[0,1])\times[0,1]$. We call the contact domain $(W\times[0,1], \lambda+dz)$   a {\em Weinstein cylinder}. Similar to a $\sigma$-plug,  a  preliminary plug blocks all trajectories entering $W\times[0,1]\times0$ at a distance $\geq\sigma$ from the boundary of the Weinstein cylinder. However, one  has a much weaker control of the dynamics of the trajectories entering near the boundary.  Constructions of 2-dimensional and preliminary plugs are variations of similar constructions in  \cite{HH}.

Next, we show in Section \ref{sec:quasi-plug} that by a special arrangement  of Weinstein cylinders $(V_1=W_1\times[0,1],\dots, V_k=W_k\times[0,1])$, $k\geq 2$, see the definition of a {\em good position} in Section \ref{sec:good-pos}, and  by  composing   preliminary plugs over $V_j\times[\frac{j-1}k,\frac jk]$, $j=1,\dots, k$,   we   create, see Lemma \ref{lm:good-positions-cor}, a {\em $\sigma$-quasi-plug} over $(\wh V:=\bigcup\limits_1^k V_j)\times[0,1]$  which   blocks trajectories entering at a distance $\geq \sigma$ from $\p\wh V\times 0$, while a  non-blocked trajectory which enters at a point  $(p_0,0)\in\wh V\times 0$  with $\dist(p_0,\p\wh V)<\sigma$   exits at a  point $p_1\in \wh V\times 1  $  which satisfies the following condition: there exist  points  $p_0',p_1'\in\p\wh W$ such that $\dist(p_0,p_0'),\dist(p_1,p_1')<\sigma$ and $p_1'$ belongs to the forward trajectory of $p_0'$ for   a vector field directing the characteristic foliation $\ell_{\p\wh V}$ on $\p\wh V$. We note that if the characteristic foliation $\ell_{\p\wh V}$ is a $\sigma$-short, then any $\sigma$-quasi-plug  is automatically a $3\sigma$-plug.

Crucial Proposition \ref{prop:3-blocks} asserts   that  a contact domain $V$ with a Weinstein convex boundary  $\p V$ and a dividing  set $S\subset \p V$ can be $C^0$-approximated by standard contact balls which coincide with $V$ in $\Op S$ then $V$ can be approximated by  3 Weinstein cylinders in a good position.

In Section \ref{sec:making-short} we use  the induction hypothesis  to show  that the standard contact ball  $D^{2n-1}$ can be be deformed by a  $\sigma$-small in the $C^0$-sense isotopy to a ball $\wt D$  with Weinstein convex boundary and  a dividing set  $S\subset \p\wt D$  such that  the characteristic foliations $\ell_{\p \wt D}$  is $\sigma$-short  and   $\wt D$    can be  $C^0$-approximated by standard contact balls which coincide with $\wt D$ on a neighborhood of $S$.

Together with Proposition \ref{prop:3-blocks} and Lemma \ref{lm:good-positions-cor} this leads in Section \ref{sec:proof-main} to a  proof of Proposition \ref{prop:main-plug-weak}, and of all main results of the paper with it.

  \section{Reducing the height of a $\sigma$-plug}\label{sec:plug-height}

  The goal of  this section  is to reduce  Proposition \ref{prop:main-plug}  to the following weaker statement.
 \begin{prop}\label{prop:main-plug-weak}
 For any $\sigma>0$  there exists  $K=K(\sigma)$ and an    isotopy
 $ h_{s}: Q\to  U_K$, $s\in [0,1]$, which satisfy properties a) and b) from Proposition \ref{prop:main-plug}.
 In other words, one can install a $\sigma$-plug over $Q=D^{2n-1}\times[0,1]$ of height $K$ which may depend on $\sigma$.
 \end{prop}
   \subsection{Changing the base} \label{sec:base}

 Let $(V_1,\Ker\alpha_1), (V_2,\Ker\alpha_2)$ be two  contact manifolds  with boundary with corners endowed with contact forms and Riemannian metrics.    
 Any contactomorpohism $f:V_1\to V_2$ can be extended to a contactomorphism $$F:(V_1\times T^*\R,\Ker(\alpha_1+xdy))\to (V_2\times T^*\R, \Ker(\alpha_2+xdy)$$    by the formula $F(v,x, y)=(f(v), g(v)x, y), \; v\in V_1,\; x,y\in\R,$ where the function $g$ is defined by the equation $f^*\alpha_2=g\alpha_1$. 
   \begin{lemma}\label{lm:sigma-distortion}
 Let $h_{s,t}:V_1\times[0,1]\to V_1\times T^*[0,1]$  be an isotopy installing a $\sigma$-plug  of height $\eps$ over $V_1\times[0,1]$. Denote $C_1:=\mathop{\max}\limits_{v\in V_1}||d_vf||$, $C_2:=\max\limits_{V_1} g$. Then the isotopy $\wh h_{s}:= F\circ h_{s}\circ F^{-1}:V_2\times[0,1]\to V_2\times T^*[0,1]$ is installing a $C_1\sigma$-plug  of height $C_2\eps$ over $V_2\times[0,1]$. 
 \end{lemma}
 \begin{proof} First, we note that  $\wh h_s(V_2\times[0,1])\subset \{|x|\leq C_2\eps\}$  because
  $h_s(V_1\times[0,1])\subset \{|x|\leq  \eps\}$ by assumption.  If $X_1^1$ is  the vector field on $V_1\times[0,1]$ directing the characteristic foliation
 of the form $h_s^*(\alpha_1+xdy)$ then  the vector field $X_1^2:=df(X_1^1)$   on $V_2\times[0,1]$ directs the characteristic foliation
defined by the form $h_s^*(\alpha_2+xdy)$.
 Let $N^1_a$ and $N^2_a$   denote   metric  $a$-neighborhoods    of  $\p V_1$ in $V_1$ and $\p V_2$ in $V_2$, respectively. Then $f(N^1_\sigma)\subset N^2_{C_1\sigma}$.  Hence, all trajectories of $X_1^2$ originated in $V_2\setminus  N^2_{C_1\sigma}$ are blocked. On the other hand, the non-blocked trajectories originated in $N^2_{C_1\sigma}$  exit with a distortion for no more than $C_1\sigma$. All other properties of a $C_1\sigma$-plug installation isotopy  listed in Proposition \ref{prop:main-plug} are straightforward. 
 \end{proof}
 
 Consider a class $\D$ of $(2n-1)$-dimensional compact manifolds  with  boundary  (and possibly with corners)  
   which are contactomorphic to a domain  in the standard contact $(\R^{2n-1},dz+\lambda_\st)$ with boundary transverse to the contact vector field $\Upsilon=2\frac\p{\p z}+\sum\limits_1^{n-1} x_j\frac{\p}{\p x_j}+y_j\frac{\p}{\p y_j}.$   For instance, the standard  $(2n-1)$-dimensional contact ball belongs to $\D$.
 \begin{lemma}\label{lm:changing-base}  If  there exists a domain $V\in\D$  such that for any $\sigma>0$ one can install a plug of height  $K=K(\sigma, V)$ over $V\times[0,1]$, then for  any domain $V'\in \D$ and   any $\sigma>0$ one can install a plug of height  $K':= K(\sigma, V'))$ over $V'\times[0,1]$.      If  there exists a domain $V\in\D$  such that for any $\sigma>0$ and any $\eps>0$  one can install a plug of height  $\eps$ over $V\times[0,1]$, then the same is true for any domain $V'\in \D$.
   \end{lemma}
 \begin{proof} For   any domain
  $( V, \Ker\alpha)\in\D$ its interior $\Int V $ is contactomorphic to the standard contact $\R^{2n-1}$, see \cite{El17}.  Besides, the boundary $\p V$ is convex, and hence $\Int V=\bigcup V_j$, where $V_j$ is contactomorphic to $V$, Hence,    for any $V,V'\in\D$   $\Int V' =\bigcup V_j$, where $V_j$ is contactomorphic to $V$. Hence, the statement follows from
  Lemma \ref{lm:sigma-distortion}.
 \end{proof}
 
 \subsection{Scaling}\label{sec:scaling}

  For $a, b>0$ denote $R_{a, b}:=\{|x_j|, |y_j|\leq a, |z|\leq b,\; j=1,\dots, n-1\}\subset (\R^{2n-1},dz+\lambda_\st)$.   Note that   $R_{a, b}\in \D$.

     \begin{figure}[ht]
 
\includegraphics[scale=0.7]{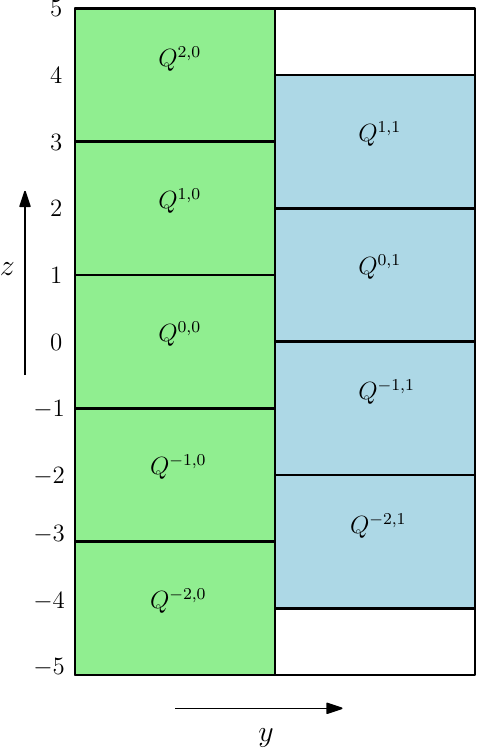}
 
\caption{The arrangement of  blocks  $Q^{i,0}$, $i=0,\pm1,\pm 2$, and $Q^{i,1}$, $i=-2,\dots, 1,$ for the case $N=2$. }
\label{fig:scaling_z}
\end{figure}

\begin{lemma}\label{lm:enlarging-height}
Choose    $\sigma>0$.  Suppose that one can install a $\sigma$-plug of height $K$  over  $R_{1,1}\times [0,1]$. Then  for any integer $N\geq 0$ one can install a $\sigma$-plug  of height $ K$    over $R_{1,2N+1}\times[0,2]$.
\end{lemma}

\begin{proof} Denote   $\wt Q:= R_{1,2N+1}\times[0,2]$, $\wh U:=R_{1,2N+1}\times T^*[0,2]$. We assume $\wh U$ is endowed with the contact form $dz+\lambda_\st+xdy$.
Furthermore,  for $i=0,\pm 1,\dots, \pm N$ denote $ Q^{i,0}:=\{z\in[2i-1 ,2i+1], y\in[0,1]\}\subset \wh Q$, $  U^{i,0}:=\{z\in[2i-1,2i+1], y\in[0,1]\}\subset \wh U$,   and  for $i=-N,\dots, N-1$  denote
  $  Q^{i,1}:=\{z\in[2i,2i+2], y\in[1,2]\}\subset \wh Q$,  $  U^{i,0}:=\{z\in[2i,2i+2], y\in[1,2]\}\subset \wh U$, see Fig. \ref{fig:scaling_z}.   Note that $Q^{0,0}=R_{1,1}\times [0,1]$   and  $U^{0,0} =R_{1,1}  \times T^*[0,1]$. The  diffeomorphisms  $(z,y)\mathop{\mapsto}\limits^{\Pi^{i,0}} (z+2i,y)$ and $(z,y)\mathop{\mapsto}\limits^{\Pi^{i,1}} (z+1+2i,y+1)$    preserves the contact form $dz+\lambda_\st+xdy$ and identify
$ (U:=R_{1,1}  \times T^*[0,1], Q:= R_{1,1}\times [0,1])$ with $(U^{i,0}, Q^{i,0})$ and $(U^{i,1}, Q^{i,1})$, respectively.  Let $h_s:Q\to U$ be an isotopy installing a plug  of height $K$
over $Q=R_{1,1}\times [0,1]$. Then the isotopy  $g_s:\wh Q\to\wh U$  which is equal to $\Pi^{i,0}\circ g_s\circ(\Pi^{i,0})^{-1} $ on $Q^{i,0}$, $i=0,\dots, \pm N$,  and to $\Pi^{i,1}\circ g_s\circ(\Pi^{i,1})^{-1} $ on $Q^{i,1}$, $i=-N,\dots, N-1$ is installing the required  $\sigma$-plug  of height $ K$   over $R_{1,2N+1}\times[0,2]$.
\end{proof}

\begin{lemma} \label{lm:scaling-plug}
 If one can install a $\sigma$-plug of height $ K $  over $ R_{1,1}\times[0,1]$ then  for any integer $N\geq 0$ one can install a $\frac\sigma {2N+1}$-plug  of height $ \frac{2K}{(2N+1)^2}$    over $R_{\frac1{2N+1},1}\times[0,1]$. 
\end{lemma}
\begin{proof} By applying Lemma \ref{lm:enlarging-height} we install a $\sigma$-plug  of height $ K$   over $R_{1,(2N+1)^2}\times[0,2]$. This is equivalent to a   $\sigma$-plug  of height $ 2K$   over $R_{1,(2N+1)^2}\times[0,1]$ . Let $h_s$  be an isotopy  which installs this plug. Consider a contactomorphism
\begin{align*}
&(x_1, y_1,\dots, x_{n-1}, y_{n-1},  x,y,z)\\
&\mathop{\mapsto}\limits^f  \left(\frac{x_1}{2N+1}, \frac{y_1}{2N+1},\dots,\frac{ x_{n-1}}{2N+1}, \frac{y_{n-1}}{2N+1}, \frac{x}{(2N+1)^2},y, \frac{z}{(2N+1)^2}\right).
\end{align*} Then the isotopy $f\circ h_s\circ (f)^{-1}$ is  installing the required   $\frac\sigma {2N+1}$-plug  of height $ \frac{2K}{(2N+1)^2}$    over $R_{\frac1{2n+1},1}\times[0,1]$. 
\end{proof}
 
  Lemma \ref{lm:scaling-plug} and  Proposition \ref{prop:main-plug-weak}  imply
  \begin{corollary}\label{cor:tiny-sigma-plug} For any $\sigma,\eps >0$ and $p\in D^{2n-2}$ there exists $N_0$ such that  for any $N\geq N_0$ one can install a  
$\frac\delta N$-plug of height $\eps$   over $R_{\frac1{N},1}\times[0,1]$.
\end{corollary}

For a point $p=(a_1,b_1,\dots, a_{n-1}, b_{n-1})\subset\R^{2n-2}$ consider a map $$\tau_p: \R^{2n-1}=\R^{2n-2}\times\R\to \R^{2n-2}\times\R$$  given by the formula
\begin{align*}
&\tau_p(x_1,y_1,\dots, x_{n-1}, y_{n-1},z)\\&=\left(x_1+a_1, y_1+b_1, \dots, x_{n-1}+a_{n-1}, y_{n-1}+ b_{n-1}, z-\sum \limits_1^{n-1}a_j y_j-b_jx_j\right).
\end{align*} Note   the $\tau_p$  preserves the contact form $\lambda_\st+dz$:
 $$\tau_p^*(\lambda_\st+dz)=\lambda_\st+dz.$$ 
  
 Fix an integer $N>1$. Given     an integer vector  $$I:=(i_1, j_1,\dots, i_{n-1}, j_{n-1})\in[1-N, N-1]^{2n-2}$$ denote $p_I =\frac IN, \;\delta_N:= \frac{2n-2}N $  and  \begin{align*}
    &P_I:=\tau_{p_I} \left(R_{\frac1{N},1}\right), \;
    P_{I, -}:=
    \tau_{p_I} \left(R_{\frac2{3N},1-\delta_N}\right),\;\;\hbox{see Fig. \ref{fig:scaling2}.}
    \end{align*}
     \begin{figure}[ht]
\includegraphics[scale=0.65]{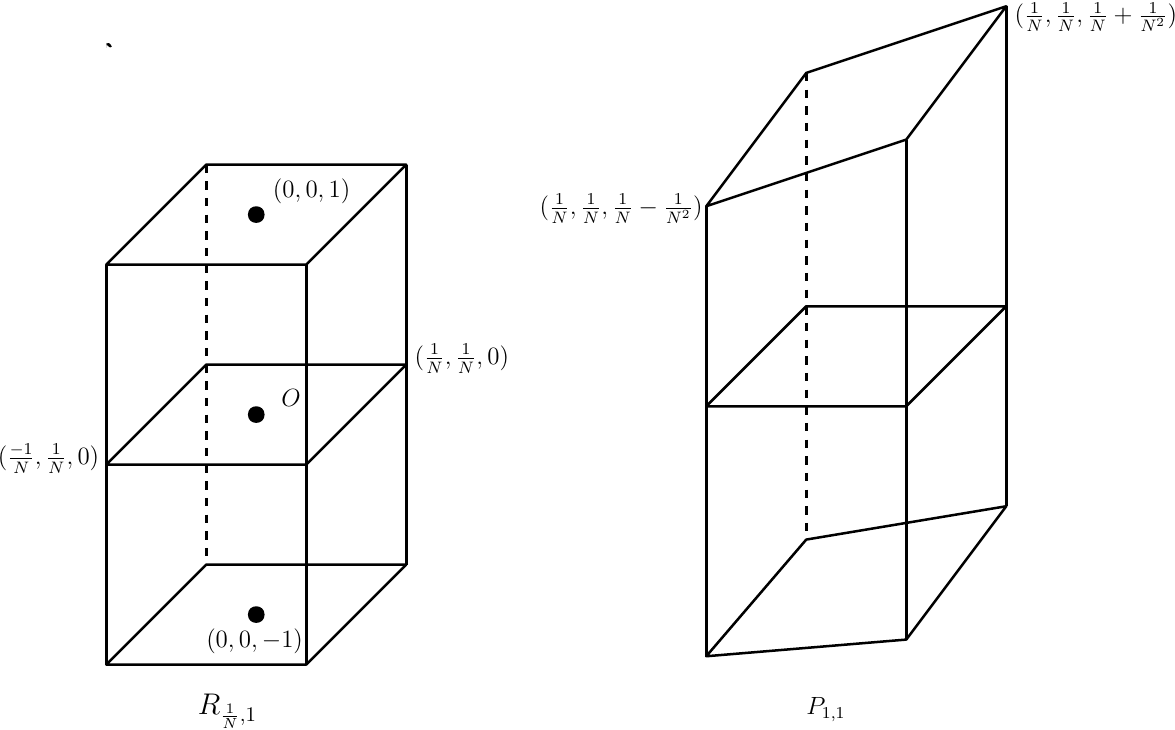}
\caption{Domain $R_{\frac1N,1}$ and its image $P_{1,1}$ under the shear $\tau_{p_{11}}$.}
\label{fig:scaling2}
\end{figure}
  Note  that we have
  \begin{lemma} $$\Int\left(R_{1-\frac1{3N},1-\frac{2n-2}N}\right) \subset\bigcup\limits_{I \in[1-N, N-1]^{2n-2}} \Int\left(P_{I,-}\right))\subset\Int\left(R_{1,1}\right) $$   and the multiplicity of the covering $\bigcup\limits_{I \in[1-N, N-1]^{2n-2}} \Int\left(P_{I,-}\right))\supset \Int\left(R_{1-\frac1{3N},1-\frac{2n-2}N}\right)$ is equal to $2^{2n-2}$.
  \end{lemma}

 \begin{proof}[Reduction of  Proposition \ref{prop:main-plug} to Proposition \ref{prop:main-plug-weak}]
 Suppose  that for any $\sigma>0$ one  can install a $\sigma$-plug of some height $K=K(\sigma)$ over $D^{n-1}\times[0,1]$. 
As the domain $R_{1,1}$ belongs to the class $\D$ it follows  that  for any $\sigma>0$ one  can install a $\sigma$-plug of some height $K'=K'(\sigma)$ over $R_{1,1}\times[0,1]$. 

 We can assume $\sigma<\frac14.$
 Set $\wh\sigma:=\frac\sigma{2^{2n-2}\sqrt{2n}}, \wh\eps:=\frac{\eps}{2^{2n-2}}$. 
 Let $N$ be the integer  provided by  Corollary \ref{cor:tiny-sigma-plug} for the pair $(\wh\sigma,\wh\eps)$. 
 In other words,  one can install a $\frac{\wh\sigma}N$-plug of height $\wh\eps$ over  $R_{\frac1{N},1}\times[0,1]$.
Choosing $N$ large enough we will ensure that $\frac{2n}N< \sigma$ which implies that $\dist(\p R_{1-\frac1{3N},1-\frac{2n-2}N}, \p R_{1,1})<\sigma$.
 Note that  $||d\tau_p||\leq \sqrt{2n}$ for any $p=(a_1,b_1,\dots, a_{n-1}, b_{n-1})\in[-1,1]^{2n-2}$.  Hence, we can apply
 Lemma \ref{lm:sigma-distortion} to install a  $\frac{\sqrt{2n}\wh\sigma}N$-plug of height $\wh\eps$ over $P_I\times[0,1]$ for any $I$. 

 Let us partition  the set  $\I=[1-N,N-1]^{2n-2}$ of all indices into $2^{2n-2}$ subsets $\I_A$ indexed by subsets $A\subset\{1,\dots, 2n-2\}$: the subset $\I_A$ consists of  $I=(i_1,j_1,\dots, i_{n-1}, j_{n-1})$ which have    odd entries  at positions of the subset $A$, and even at other places.  For instance, for $A=\varnothing$ the set $\I_{\varnothing}$ consists of  $I=(i_1,j_1,\dots, i_{n-1}, j_{n-1})$, where all $i_k, j_k$ are even.  Note that for $I,I'\in \I_A$, $I\neq I'$ we have $\Int P_I\cap \Int P_{I'}=\varnothing$. We  enumerate all  subsets $A\subset\{1,\dots, 2n-2\}$  as $A_1,\dots, A_{2^{2n-2}}$, and write $\I_j$ instead of $\I_{A_j}$. 

We claim that by installing for  each  $I\in\I_j$, $j=1,\dots, 2^{2n-2}$   a 
  $\frac{\sqrt{2n}\wh\sigma}N$-plug   of height $2^{2n-2}\wh \eps=\eps$ over $P_I\times[\frac{j-1}{2^{2n-2}}, \frac{j}{2^{2n-2}}]$ we construct the required $\sigma$-plug of height $\eps$ over $R_{1,1}\times[0,1]$.
   Indeed,   let  $h_s:R_{1,1}\times[0,1]\to R_{1,1}\times T^*[0,1]$ be the resulting isotopy. First, note that $h_s(R_{1,1}\times[0,1])\subset \{|x|\leq 2^{2n-2}\wh\eps=\eps$.  Let us verify that the vector field $X_1$ directing the characteristic foliation $\ell_1$ induced by $h_1^*(\alpha_\st+xdy)$ is a $\sigma$-plug.
  As there are    $2^{2n-2}$ layers of plugs,  each  trajectory $\gamma$ of $X_1$ beginning at $(p,0)\in R_{1,1}\times 0$ intersects no more than $2^{2n-2}$ plugs $ P_I\times[\frac{j-1}{2^{2n-2}}, \frac{j}{2^{2n-2}}]$. Each of these plugs either blocks $\gamma$, or displaces it for no more than $\frac{\sqrt{2n}\wh\sigma}N$. Hence  if $\gamma$ exits through a point $(p',1)\in R_{1,1}\times 1$ then $\dist(p,p')< 2^{2n-2}\sqrt{2n}\frac{\wh\sigma}N=\frac\sigma N$. On the other hand, if $\dist(p,\p R_{1,1})\geq\sigma>\frac{2n}N$ then $p\in P_{I,-}$ for a multi-index $I\in\I_j$ for some $j=1,\dots, 2^{2n-2}$. If the trajectory $\gamma$ originates at $(p,0)$ and it is not blocked by any of the plugs on the layers $ [\frac{i-1}{2^{2n-2}}, \frac{i}{2^{2n-2}}]$ for $i<j$ then by the above argument it enters the plug $P_{I,-}\times [\frac{j-1}{2^{2n-2}}, \frac{j-1}{2^{2n-2}}] $ through a point $(p',  \frac{j-1}{2^{2n-2}})$ with $\dist(p,p')<\frac\sigma N$. Hence, $\dist (p',\p P_I)>\dist(\p P_I,\p P_{I,-})-\frac\sigma N>\frac1{3N}-\frac1{4N}>\frac{\sigma}{3N}>\frac{\sqrt{2n}\wh\sigma}N$. But   $(P_{I,-}\times [\frac{j-1}{2^{2n-2}}, \frac{j-1}{2^{2n-2}}], X_1) $  is a $\frac{\sqrt{2n}\wh\sigma}N$-plug, and therefore, the trajectory $\gamma$ is blocked inside $P_{I,-}\times [\frac{j-1}{2^{2n-2}}, \frac{j-1}{2^{2n-2}}] $.
 This verifies the property a) of  Proposition \ref{prop:main-plug}. Property b) follows from the fact that it holds for plugs  $(P_{I,-}\times [\frac{j-1}{2^{2n-2}}, \frac{j-1}{2^{2n-2}}],X_1) $ for each $I\in\I$ and transversality arguments. 
 
 Finally, we again apply Lemma \ref{lm:changing-base} to conclude that the installation for any $\sigma,\eps$ of a $\sigma$-plug of height $\eps$ over $R_{1,1}$ is equivalent to the   installation for any $\sigma,\eps$ of a $\sigma$-plug of height $\eps$ over $D^{2n-1}\times[0,1]$, because both domains $D^{2n-1}$ and $R_{1,1}$ belong to the class $\D$.
    \end{proof} 
    \begin{remark}\label{rem:2d}
    Note that the above proof of the height-reduction  for   $\sigma$-plugs significantly simplifies for $n=1$, i.e. when the plug is 2-dimensional.
    Indeed, in this case $D^{2n-1}=R_{1,1}=R_{\frac1N,1}=[-1,1]$, and the claim follows  directly from Corollary \ref{cor:tiny-sigma-plug}.
    \end{remark}
    
  \section{The 2-dimensional case}\label{sec:2d-plug}
We will prove in this section Proposition \ref{prop:main-plug-weak} (and hence,  Proposition \ref{prop:main-plug}) in the 2-dimensional case. In fact, we will establish a stronger statement, Lemma \ref{lm:2d-plug}, which will enable us to continue the construction by induction on dimension of the plug.

 \subsection{Creation and elimination of singularities of  a 2-dimensional characteristic foliation}\label{sec:GF}
 
The following statement is a slight modification of  the  Giroux-Fuchs elimination lemma, see \cite{Gi}.
 
 \begin{lemma} \label{lm:creation-elimination}
 Let $\Sigma$ be a 2-dimensional surface in a contact $3$-manifold $ (M,\xi=\Ker\alpha)$.
 Let $p\in\Sigma $ be a non-singular point of the characteristic foliation $\ell=\ell_{\Sigma,\xi}$. Let $\gamma\ni p$ be an arc of the leaf of $\ell$ through $p$.   Suppose that $(d\alpha)_p>0$
Then for any positive  $\eps$ and $\sigma\ll \eps$ there exists an $\eps$-small 2-parametric  isotopy $\phi_{t,s}:\Sigma\to (M,\xi)$ supported in an $\eps$-neighborhood of $p\in\Sigma$ with the following properties. Denote $\beta_{s,t}:=\phi_{s,t}^*\alpha$ and let $\ell_{s,t}$ be the characteristic foliation defined by $\beta_{s,t}$, and $X_{s,t}$ the vector field directing $\ell_{s,t}$.
 \begin{itemize} 
 \item[-] $\phi_{0,t}=\phi_{0,0}$ is the inclusion $\Sigma\hookrightarrow M$;
 \item[-] the $1$-form $\beta_{s,1}$ has  no zeros for all $s\in[0,1]$;
 \item[-] the $1$-form $\beta_{1,0}$ has exactly two zeros, one positive elliptic and one hyperbolic on the arc $\gamma$;
 \item[-] the arc $\gamma$  is tangent to $X_{s,t}$ for all $s,t\in[0,1]$.
 \item[-] $d\beta_{s,t}=d\beta_{0,0}$ for all $s,t\in[0,1]$;
 \item[-] If $\ell$ admits a Lyapunov function $f:\Sigma\to\R$ then for all $s,t\in[0,1]$  the characteristic foliation $\ell_{s,t}$    admits a Lyapunov function  $f_{s,t}$ which coincides with $f$ outside an $\eps$-neighborhood of $p$.
 \end{itemize}
 \end{lemma}

 Let us first observe that   Proposition \ref{prop:AG} allows us to normalize  the contact structure in a neighborhood of a points $p\in\Sigma$.
  \begin{lemma}\label{lm:normal-form-non0}
 Consider the form $\delta=dz+dx+xdy$ in $\R^3$.
 Under assumptions of Lemma \ref{lm:creation-elimination} there is a neighborhood $U\ni p$ in $M$,
 a neighborhood $U'$ of $0$ in $\R^3$ and   a contactomorphism $h:(U',\delta)\to (U,\alpha)$  
such that $h^{-1}(\Sigma\cap U)=\R^2\cap U'$.
\end{lemma}
Hence, Lemma \ref{lm:creation-elimination} follows from the following statement.

 \begin{lemma}\label{lm:GF-model}
 For any $\eps\gg\sigma>0$ there exists a 2-parametric  family of $C^\infty$-functions $G_{s,t}:\R^2\to\R$, $s,t\in[0,1]$ which are supported in  $\{|x|,|y|<\eps\}$ which have the following properties. Denote
 $\alpha_{s,t}:=xdy+dx+dG_{s,t}.$
 \begin{itemize}
 \item[-] $G_{0,0}=0$;
 \item[-] $G_{s,t}-G_{s,0}$ is supported in $\{x<\sigma, y<\eps\}$.
 \item[-] the $1$-form $\alpha_{s,1}$ has  no zeros for all $s\in[0,1]$;
 \item[-] the $1$-form $\alpha_{1,0}$ has exactly two positive zeroes, one elliptic and one hyperbolic, on the line $\{x=0\}$.
 \item[-] for all $s,t\in[0,1]$ the vector field $Y_{s,t}$  directing the characteristic foliation on $\{x = 0\}$ generated by $ \alpha_{s,t}$ admits a Lyapunov function which is equal to $y$ outside of a compact set.
 \end{itemize}
 \end{lemma}
   
  \begin{figure}[ht]
 
\includegraphics[scale=0.45]{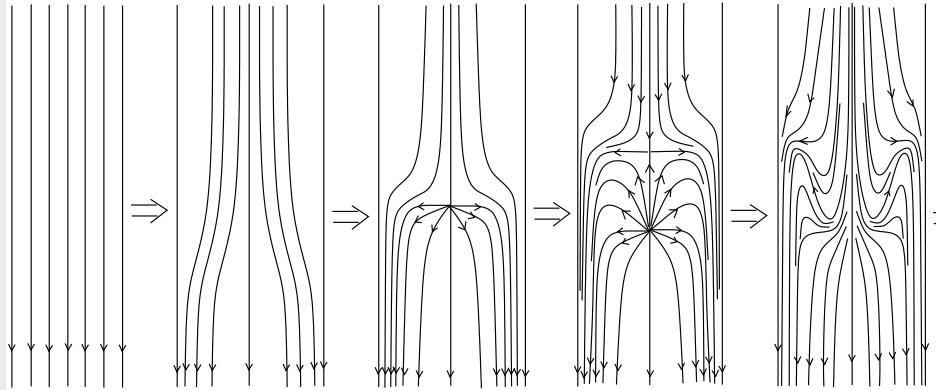}
 
\caption{Creating   and eliminating  zeroes.}
\label{fig:birthdeath}
\end{figure}

 \begin{proof}   
 Consider  an even    function  $\theta :\R \to\R_+$, 
 $\theta(u)=\theta(-u)$, and for any $c>0$ consider an odd  functions $\eta_c:\R\to\R, \; \eta_c(u)=-\eta_c(u)$
 such that   the following properties are satisfied:
 
 \begin{itemize}
 \item[-]$\theta(u)= u^2-1-\frac{\eps^2}9$   on $u\in[-\eps/2,\eps/2]$,  $\theta (u)=0$ for $ |u|\geq \eps$  and $0\leq\theta'(u)$ for $u\geq 0$;
 \item[-] $\eta_c(u)=u$  for  $u\in[-c\sigma, c\sigma]$, $\eta_c(u)=0$  for $ |u|\geq c\eps$, $|\eta_c(u)|\leq |u|$ and $- \frac{2\sigma}\eps\leq    \eta'(u)\leq 1$ for $u\geq 0$.
 \end{itemize}
Consider  a family of functions $G_{s}, \; s\in[0,1]$ by the formula
 \begin{align*}G_{s }(x,y)= s \theta(y )\eta_1 (x ) \end{align*}
  and a family of 1-forms
  \begin{align*} \alpha_s&=dx+xdy+dG_s= (1+ s \theta ( y )\eta_1' (x) )dx 
   + (x+s\theta'(y)\eta_1(x)) dy =f_s(x,y)dx+g_s(x,y)dy.
  \end{align*}

  Let us check  that the form $\alpha_1$ have a hyperbolic $0$ at the points $(x=0, y=-\frac\eps 3)$, an elliptic zero at $(x=0, y=\frac\eps 3)$ and no other zeroes.
  We have    \begin{equation}
 \label{eq:1-param}\begin{split}
  & |g_1(x,y)|=|x+2y\eta_1(x)|\geq|x|(1-2|y|)\geq(1- \eps )|x|\;\;\hbox{for}\;\; |y|\leq\frac\eps 2 ;\\
  &f_1(x,y)=1+  \left(y^2-1-\frac{\eps^2}9\right)\eta_1' (x) \geq  \frac{5\eps^2}{36}>0,\;\;  |y|\geq\frac\eps 2.
  \end{split}
  \end{equation}
  
   Hence $\alpha_1$ has only zeros along the interval $\{x=0, -\frac\eps2<y<\frac\eps2\}$. In the neighborhood of this interval 
we have $\alpha_1=(y^2 -\frac{\eps^2}9)dx+ (1+2y)x dy$, which has 2 zeroes, elliptic and hyperbolic, respectively  at the points  $(x=0, y=\frac\eps 3)$ and   $(x=0, y=-\frac\eps 3)$.

Let us now extend the family $G_s$ to the 2-parametric family of functions
$G_{s,t}:\R^2\to\R$ by setting
$$G_{s,t}=H_s(x,y)-st\theta(y)\eta_\sigma(x)=f_{s,t}(x,y)dx+f_{s,t}(x,y)dy.$$
Let us verify that the form
$\alpha_{s,1}=dx+xdy+ dG_{s,1}$ has no zeros.
First, note that for $|x|\leq \frac{\sigma^2}\eps,$ we have
$f_{s,1}=1$ and 
$g_{s,1}(x,y)=x$.

Similarly to the above estimates \eqref{eq:1-param} for $f_s$ and $g_s$ we conclude
that $f_{s,1}\neq 0 $ for $|y|\geq \frac{\eps}2$ and $x\neq 0$  and $g_{s,1}\neq 0$  $|y|\leq \frac{\eps}2$.

It remains to show  
 existence of a family of Lyapunov functions for  the family of   vector fields $Y_{s,t}.$   According to Corollary  \ref{cor:good-Lyapunov} it is sufficient to verify  for $Y_{s,t}$ the property (L1)  and the Morse-Smale condition.
 Because it is the 2-dimensional case then by  Poincar\'e-Bendixson's theorem it is sufficient to show  that there are no periodic orbits. But any periodic orbit in $R$ bounds a disc and the sum of indices of singular points in this disc should be equal to $1$. On the other hand,  the only 2 singular points of $Y_{s,t}$ are connected by a separatrix trajectory, and hence the disc bounded by  a  periodic orbit must enclose both singular points, whose sum of indices is equal to  $0$. 
   \end{proof}

 \subsection{Special 2-dimensional plug}\label{sec:2-plug}
  
  \begin{figure}[ht]
 
\includegraphics[scale=0.8]{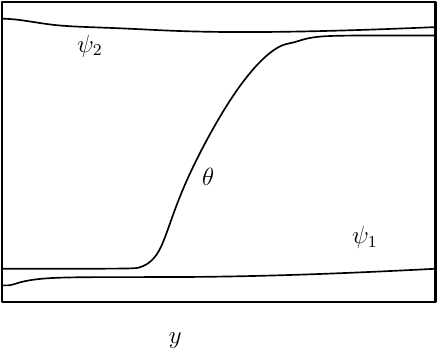}
 
\caption{Graphs of functions  $\Theta$, $\Psi_1$,  and $\Psi_2$. }
\label{fig:graphs}
\end{figure}

We  construct in this section  a  special 2-dimensional plug. 
In  the  contact  space $(\R^3, \Ker\{dz+xdy\})$ consider
  $$ O :=\{0\leq y, z\leq 1,  -4\leq x\leq 0\} ,\;\;R:= O \cap\{x=0\}.$$
For a   sufficiently small $\eps$  let us choose     non-decreasing  $C^\infty$-functions $\psi, \theta:[0,1]\to\R$ such that
\begin{itemize}
\item  $ \theta(y)=  \eps$ for $y\in[0,\frac13]$; $ \theta(y)=1- \eps$ for $y\in[\frac23,1]$ and 
$0<\theta'(y)<4$ for $y\in(\frac13,\frac23)$.
\item  $\psi(0)=0$, $\psi_1(1)<\frac\eps2$;
\item $\psi$  has  vanishing   derivatives of all orders at the points $0$ and $1$, and  $0< \psi'(y)< \eps$ for $y\in(0,1)$. 
\end{itemize} 
  Denote $$\psi_1:=\psi+\frac\eps2,\;
   \psi_2:=-\psi+(1-\psi(1)-\frac\eps2).
   $$ and consider the graphs
 $\Theta, \Psi_1, \Psi_2$   of the functions $\theta,\psi_1$ and $\psi_2$:
  \begin{align*}
  &\Theta:=\{(y,\theta(y));\;y\in[0,1]\},\;
  \Psi_1:=\{(y,\psi_1(y)); \; y\in[0,1]\},\\
 & \Psi_2:=\{(y,\psi_2(y)) ;\; y\in[0,1]\}\subset R.
  \end{align*}
 Denote $$R_+:=\{(y,z)\in R, z\leq\theta(y)\}, \; R_-:=\{(y,z)\in R, z\geq\theta(y)\}.$$
 
\begin{lemma}\label{lm:2d-plug} For any $  \eps>0$
there exists an      isotopy 
$h_s:R\to O$, $s\in[0,1] $, which is fixed together with its $\infty$-jet along $\p R$,   constant for $s\in[0,\frac18]$,  and such that the following properties  a)--i) are satisfied.
Denote $\beta_s:=h_s^* (dz + x dy)$. Let $Y_s$ be the vector field directing the characteristic foliation $\ell_s$ of $\beta_s$.  
\begin{itemize}
   \item[a)] $d\beta$ restricted to the interior $ \Int R_{+}$ of $R_ +$  is positive, and  $ d\beta$ restricted to the interior $\Int R_{-}$ of $R_-$ is negative
for all   $s\in(\frac18,1]$;
\item[b)] $\beta_s$ has
\begin{itemize}
\item no zeros    for $s<\frac34$,
\item  a positive and negative embryos $o_+:=(\frac16,\psi_1(\frac16)),
o_-=(\frac56,\psi_2(\frac56))$ for $s=\frac34$,
 \item a pair 
 \begin{align*}
 &e_+(s)=(e_+^1(s),\psi_1(e_+^1(s))), \hbar_+(s)=(\hbar_+^1(s),\psi_1(\hbar_+^1(s))),\; 0<e_+^1(s)<\hbar_+^1(s)<\frac13
 \end{align*} of positive elliptic and  hyperbolic   points, and a pair
 \begin{align*}
  &e_-(s)=(e_-^1(s),\psi_1(e_-^1(s))), \hbar_-(s)=(\hbar_-^1(s),\psi_1(\hbar_-^1(s))),\; 1>e_+^1(s)>\hbar_+^1(s)>\frac23
 \end{align*} of negative elliptic and  hyperbolic points for $s>\frac34$.
\end{itemize}
 
\item[c)] the incoming separatrices of $\hbar_+(s)$   and $o_+$  for $Y_s$,    are  contained in $\Psi_1$, and outgoing separatrices
of $\hbar_-(s)$   and $o_-$  for $Y_s$ are  contained in $\Psi_2$,  $s\geq 
\frac34$;
 \item[d)] there exists $\eps_1\in(0,\eps)$ such that the outgoing separatrices of $\hbar_+(s)$ for $Y_1$ terminate at $e_-$ and $(1,\eps_1 )$, and the   incoming separatrices of $\hbar_-(s)$ for $Y_1$ originate at $e_+$ and $(0,1-\eps_1)$; 
\item[f)]   $Y_s$  for $s\in[\frac12,1]$ is outward transverse  to the graph 
$\Theta$, viewed as a  part of the boundary of the domain $R_+$;
\item[g)]   $Y_{s}$ admits   a  family of good Lyapunov function  $\psi_{s}:R\to\R$   such that  $\psi_{s}|_{\Op\p R}=y$;
\item[h)] $\beta_s(\frac{\p}{\p z})>0$  everywhere in $R$ for $s\in[\frac12,1]$;
\item[i)]  for any $\sigma>0$  the isotopy $h_s$, $s\in[0,1]$,  can be extended to a  2-parametric isotopy $h_{s,t}, 0\leq s,t\leq1$,  such that
\begin{itemize}
\item $h_{s,0}=h_s$, $h_{0,t}=h_0$ for all $s,t\in[0,1]$;
\item $h_{s,t}=h_{s}$   for $s\leq\frac34-\sigma, t\in[0,1]$;
\item for each $s\in(\frac34,1]$  the isotopy $h_{s,t},t\in[0,1]$, is supported in a    $\sigma$-neighborhood of the separatrices  connecting $\hbar_\pm(s)$ with $e_\pm(s)$;   for each $s\in[\frac34-\sigma,1]$ $h_{s,t}$ is supported in a    $\sigma$-neighborhood of $o_\pm$;  
 \item $h_{s,t}$ is   $\sigma$-close in the $C^0$-sense to $h_{s,0}$ for all $s,t\in[0,1]$;
\item  the family of vector fields  $Y_{s,t}$ directing the characteristic foliations $\ell_{s,t}$ of $\beta_{s,t}:=h_{s,t}^*\alpha$ admits a  family  good Lyapunov functions  $\psi_{s,t}:R\to\R$   such that  $h_{s,t}|_{\Op\p R}=y$;
\item $Y_{s,t}$ has a pair  of positive elliptic and hyperbolic zeroes at
$e_\pm(s(1-2t) +\frac{3t}2)$ and $\hbar_\pm(s(1-2t) +\frac{3t}2)$ for $s>\frac34, t<\frac12$, pairs of embryos at $o_\pm$ for $s=\frac34, t=\frac12$ and no zeroes otherwise;
\item $h_{s,t}=h_{0}$ for $s\in[0,\frac18], t\in[0,1],$ and $h_{s,t}=h_{1,t}$ for $s\in[\frac78,1]$    $t\in[0,1]$;
\item $X_{s,t}$ is outwardly transverse to $\Theta$ for all $s\in[\frac12,1],\; t\in[0,1]$.
  \end{itemize}
 
\end{itemize}

\end{lemma}
\begin{figure}[ht]
 \begin{center}
\includegraphics[scale=0.7]{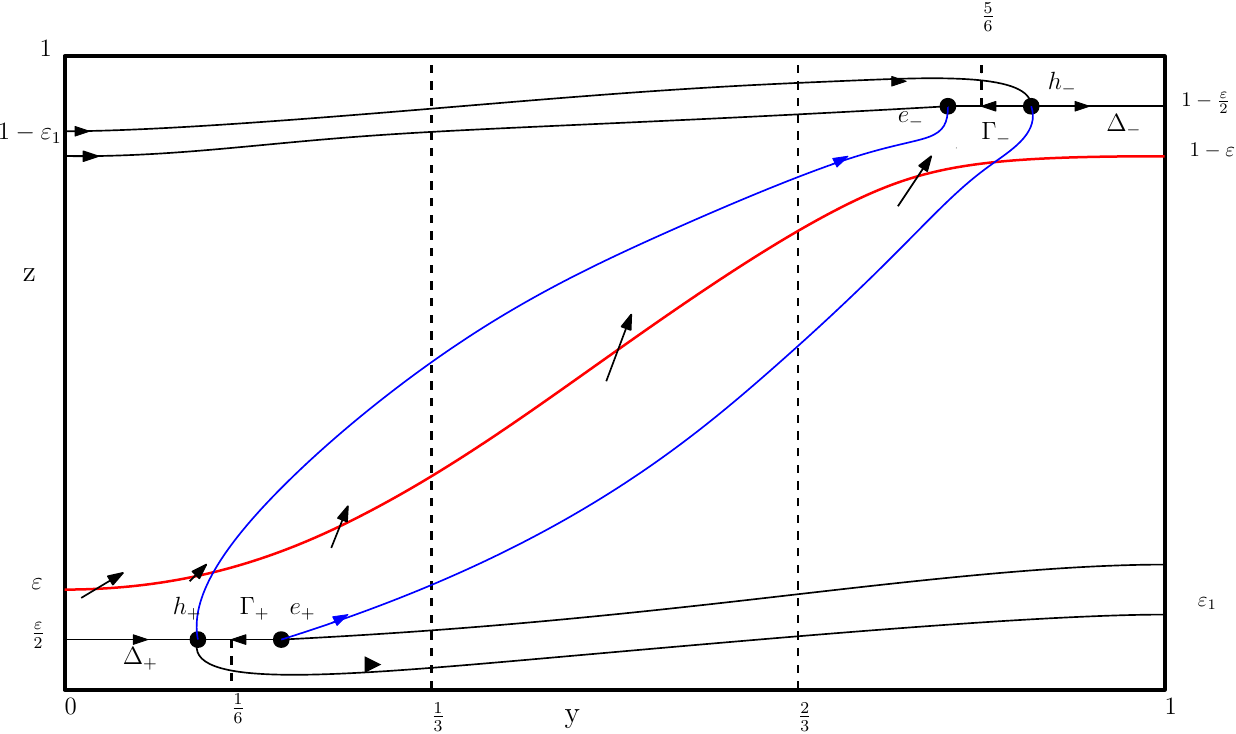}
 \end{center}
\caption{The characteristic foliation on $h_{\frac{4}{5}}(R).$ The red curve is the dividing set $\Gamma = \{d \mu = 0\}$ while blue curves
depict separatrices connecting $e^{\pm}$ to $h^{\mp}$.  }
\label{fig:char_fol}
\end{figure}

\begin{proof}
  Choose a function $H:R\to\R$ such that
  \begin{itemize}
\item[(C1)] $H$ vanishes on  $\p R$ together with all its derivatives;
\item[(C2)]   $0\geq  H(y,z)\geq -4$, $y,z\in[0,1]$;
\item[(C3)] $H(y,\psi_1(y)  )=H(y,\psi_2(y) )=-\psi'(y).$ 
\item[(C4)]   
   $\frac{\p H}{\p z}(y,z)=\begin{cases}<0, 
   &(y,z)\in R_+\\
   >0,&(y,z)\in R_-;
\end{cases}
$ 
\item[(C5)] $H(y,\theta(y))\leq-\theta'(y)$.
\end{itemize}
An additional property (C6) will be imposed  later.

Define an isotopy
$ h_s:R\to\R$, $s\in[0,1]$, as follows.  
For $s\in[0,1/2]$ we define  $$h_s(y,z):=(y,z, 2sH(y,z)),\;\;  (y,z)\in R .$$  
Let $\ell_s$ be the characteristic foliation defined by $h_s^*\wh\mu$ on $R$.
 Leaves of the characteristic foliation on $\ell_s$  are
 graphs of  solutions of the equation
\begin{equation}\label{eq:leaves}
\frac{dz}{dy}=-2sH(y,z).
\end{equation}
 
   For $\zeta\in[0,1]$  we denote by
$ \ell_{\zeta} $ and $\ell^\zeta$  the   solutions of \eqref{eq:leaves} with the initial data
$\ell_\zeta(0)=\zeta$ and $\ell^\zeta(3)=\zeta$, respectively.  Condition (C3) ensures that  $\ell_{\frac\eps2}=\psi_1 $, $\ell^{\frac\eps2}=
\psi_2$.

For $s\in[1/2,1]$ we use Lemma \ref{lm:creation-elimination} to create for $s>\frac34$ pairs of elliptic-hyperbolic positive and negative points at $e_\pm(s),\hbar_\pm(s)$ through   embryos at  $o_\pm$ for $s=\frac34$. The isotopy  can be constructed arbitrary $C^0$-small and supported in a neighborhood of separatrices connecting $e_\pm$ and $\hbar_\pm$.
It can also be arranged     that the isotopy also  fixes the leaves $\Psi_1$ and $\Psi_2$ of the foliation $\ell_{\frac12}$,    
     so that these  leaves become broken leaves of the characteristic foliation $\ell_s$ for $s\in[\frac34,1]$. In particular,   these curves contain, respectively,  the incoming separatrix of $\hbar_+(s)$  and outgoing separatrix of $\hbar_-(s)$.  It then follows that one of the outgoing separatrices  of $\hbar_+(s)$ terminates at $(1,\eps_1) $ for $\eps_1<\eps$, and it could be arranged that  one of the incoming separatrices  of $\hbar_-$ originates at $(1,1-\eps_1) $.
       
  Suppose that the second outgoing separatrix of $\hbar_+(1)$ intersects the line $y=\frac13$ at a point $(\frac13,a)$, $a\in(\psi_1(\frac13),\theta(\frac13))$, while  the second  incoming separatrix of $\hbar_-(1)$ intersects the line $y=\frac23$ at a point $(\frac23,b)$, $b\in(\theta(\frac23),\psi_2(\frac23)) $. We  now impose the remaining condition on the function $H$:
  \begin{itemize}
  \item[(C6)] 
  $H(y,z)<-3$ for $y\in[1,2], z\in[a, b]$;
  \end{itemize}
   This guarantees that one of the outgoing separatrices of $\hbar_+(1)$ terminates at $e_-(1)$, and one of the incoming separatrices of $\hbar_-(1)$ originates at $e_+(1)$.
 
   Using the extension to the 2-parametric isotopy  in Lemma \ref{lm:creation-elimination} we extend the isotopy $h_s$ to a 2-parametric isotopy $h_{s,t}$ for $s,t\in[0,1]$ with the required properties.  \end{proof}

Let us denote by $\Gamma_\pm(s)$, $s\in[\frac34,1]$ the (closure of the) trajectory of $Y_s$ connecting $e_\pm(s)$ and $\hbar_\pm(s)$ and by $\Delta_+(s)$ (resp. $\Delta_-(s))$  the (closure of the) incoming (resp. outgoing) separatrix of $\hbar_+(s)$ (resp. $\hbar_-(s)$). For $s=\frac34$ we assume that $e_\pm(\frac34)=\hbar_\pm(\frac34)=o_\pm$.  We extend the definition of $\Gamma_\pm(s)$ and $\Delta_\pm(s)$ to all $s\in[0,1]$
by setting   $\Gamma_\pm(s)=\Delta_\pm(s)=\varnothing$ for $s<\frac34$.

\section{Preliminary and quasi-plugs}\label{sec:prelim-quasi-plugs}
 \subsection{Preliminary plug}\label{sec:prelim-plug}
 Let $(W,\lambda)$ be a Weinstein domain. We denote by $Z$ the Liouville field dual to $\lambda$. Consider an interior  boundary collar $C:=\p W\times[  1-\eps,1]$ such that $\p W\times1=\p W$ and $\lambda|_C=\tau\gamma$, $\tau\in[1-\eps,1],$ for a contact form $\gamma=\lambda|_{\p W}$.
 Denote $W_0:=W\setminus C$.  Furthermore, denote by $\Skel(W,Z)$ the skeleton of $W$, i.e. the union of stable manifolds of zeroes of $Z$.
 Alternatively, $\Skel(W,Z)=\bigcap\limits_{s\in[0,\infty)}Z^{-s}(W)$. Here we denote by $Z^{-s}$ the    flow of  $-Z$, which is defined for all $s\geq 0$.
 
 \smallskip
 Consider a contact manifold $(V:=W\times[0,1], \Ker (\lambda+dz))$, and in $(V\times T^*[0,1], \Ker(\lambda+dz+xdy))$ take the domain
 $ U=\{-4\leq x\leq 0\}$ and a hypersurface $Q=\{x=0\}$. Note that $Q=V\times[0,1]$, and we can naturally identify 
$ V\times T^*[0,1]$ and $ Q$   with $W\times O$ and $W\times R$ respectively, where  we  use the notations $O$ and  $R$ from 
 Section \ref{sec:2-plug} which denoted a three dimensional cube and the rectangle $\{x = 0\}$ contained in the cube. .

Let $h_{s,t}:R\to O$ be the isotopy constructed in Lemma \ref{lm:2d-plug}. Define an isotopy $$g_{s,t}:Q=W\times R\to U=W\times O$$ by the formula
\begin{align*}
 g_{s,t}(w,q)=\begin{cases}
 (w, h_{s,t}(q)),& w\in W_0,q\in R\\
 (w, h_{s\ol\tau,t}(q)),& w=(v,\tau)\in C=\p W\times[1-\eps,1],\;
 \ol\tau=\frac{1-\tau}{\eps} 
 \end{cases}  
\end{align*}
 Denote $\wh \beta_{s,t}:=g_{s,t}^*(\lambda + dz + x dy)$ and let $X_{s,t}$ be the vector field directing the characteristic foliation defined by the form $\wh \beta_{s,t}$.    Set  $\wh\beta_s:=\wh\beta_{s,0}$ and $X_s:=X_{s,0}$.
It follows from the corresponding properties of  the isotopy $h_{s,t}$ that for each fixed  $s\in[0,1]$ the isotopy $g_{s,t}$, $t\in[0,1]$, is $\sigma$-close to $g_s$ in the $C^0$-sense.

Denote  \begin{align*}
&    {}^{in}P:=W_0\times[\eps, 1-\eps_1]\times 0 ,\;
 \;{}^{out} P:=W_0\times[\eps_1, 1-\eps]\times 1 ,\\
 &  {}^{in}T:=W_0\times(1-\eps_1,1]\times 0 , \;\;
  {}^{out}T:=W_0\times[0,\eps_1)\times 1 ;\\
  &{}^{in} S=W_0\times[0,\frac\eps2]\times0 ,\;\;
  {}^{out}S=W_0\times[1-\frac\eps2,1]\times 1
  \end{align*}
 
 Furthermore, denote
 $$\wh\Gamma_\pm:=W_0\times \Gamma_\pm(1)\cup\bigcup\limits_{u\in\p W,\;1-\eps\leq \tau\leq 1}(u,\tau)\times\Gamma_\pm({\ol\tau}) \subset Q,\;
\ol\tau=\frac{1-\tau}{\eps}. $$

\begin{figure}[ht]
 \begin{center}
\includegraphics[scale=0.7]{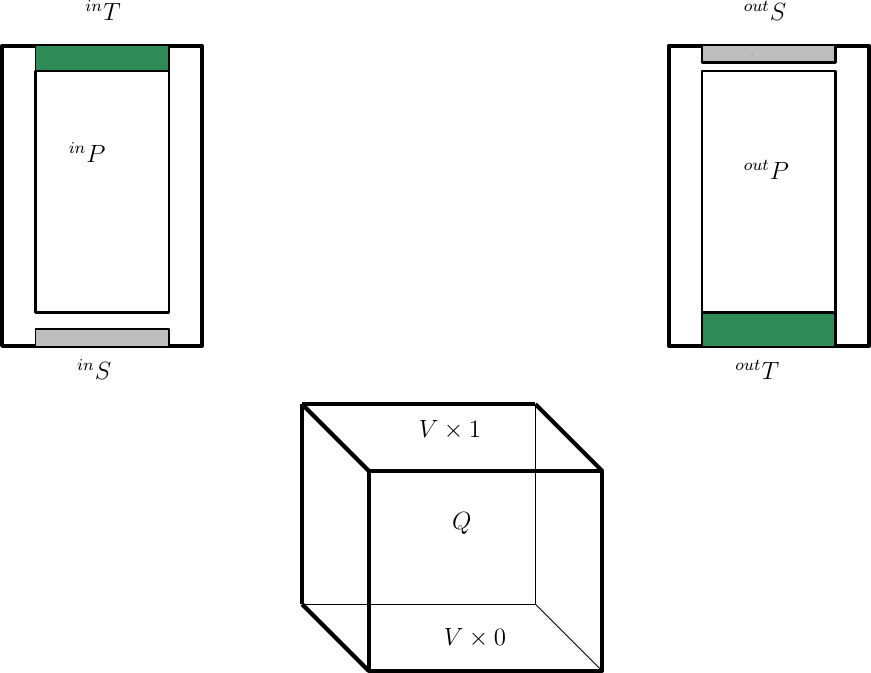}
 \end{center}
\caption{Regions in the preliminary block with  controlled dynamics.}
\label{fig:block_dynamics}
\end{figure}

The following proposition lists  properties of $X_{s,t}$ which will be  necessary for  constructions of $\sigma$-plugs. 
\begin{prop}\label{prop:prelim-plug}
 The vector field   $X_{s,t}$ on $Q$, $s,t \in[0,1] $, has the following properties.  
\begin{enumerate}
 
\item  every trajectory of $X_{1}$
\begin{enumerate}
\item  which starts  at
 ${}^{in}P  $ converges to a zero of $X_{1}$;

 \item  which starts  at a point $(u,z,0) \in {}^{in}T$    exits   at a point $(u',z',1)\in {}^{out}S   $,   where  $u'=Z^{-a}(u) , a>0$  and  $z'\in(1-\frac\eps2,1]$;  
\item   which ends 
 ${}^{out}P $ originates at a zero of $X_{1}$;
 \item which terminates  at a point $(u,z,1) \in {}^{out}T $ begins at   a point $(u',z',0)  \in {}^{in}S \subset {}^{in}V, $  where  $u'=Z^{-a}(u) , a>0$, and  $z'\in[0,\frac\eps2)$;  
\end{enumerate}
 
\item  vector fields $X_{s,1}$, $s\in[0,1] $, have no zeroes;
\item $\wh\Gamma_+$ is invariant with respect to the negative flow of $X_1$, and $\wh\Gamma_-$ is invariant with respect to the positive one;
\item all trajectories of $X_s$ which converge to positive singularities either do not intersect $\p Q$ and are contained in $\wh\Gamma_+$, or intersect it at points of $\Skel(W)\times\frac\eps2\subset {}^{in}V$; all trajectories of $X_s$ which originate at negative singularities either do not intersect $\p Q$ and are contained in $\wh \Gamma_-$, or intersect it at points of $\Skel(W)\times(1-\frac{\eps}2)\subset {}^{out}V;$
\item  the  family of vector fields $X_{s,t}$ admits a family of good Lyapunov functions $ \psi_{s,t}$, equal to $y$ on $\p Q$;
\end{enumerate}
\end{prop}  
 We begin the proof with an explicit computation of the vector field $X_{s,t}$.  Let $Z$ denote the Liouville field on $W$ corresponding to the Liouville form $\lambda$. Recall that $\lambda|_C=\tau\gamma$, $\tau\in[1-\eps,1]$, where $\gamma$ is a contact form on $\p W$. We have $Z|_C=\tau\frac{\p}{\p\tau}.$ 
 Let $R$ be the Reeb vector field on $\p W$ lifted to $C=\p W\times[1-\eps,1]$ via the projection to the first factor.
 
 Let us view  $Q=W\times R$ as the fiber bundle over $W$.
 The form $\wh\beta_{s,t}$ restricts to the fiber $w\times R,\; w\in W$ as $\beta_{s,t}$ if $w\in W_0$ and  for $w=(u,\tau)\in C=\p W\times[1-\eps, 1]$ as $\beta_{s\ol\tau,t},$ where
 $ \ol\tau=\frac{1-\tau}{\eps} $. 
 To simplify the notation we will write $\wh\mu$ instead of $\wh\beta_{s,t}$, $X$ instead of $X_{s,t}$, $\mu$ instead of
 $\beta_{s,t}$ and  $\mu_\tau$ instead of $\beta_{s\ol\tau,t}$ while writing the proof of Proposition~\ref{prop:prelim-plug}. On the other hand
 $X_1$ will continue to denote $X_{1,0}$.  Finally, we will denote
 $\dot\mu:=\frac{d\mu_\tau}{d\tau}$. 
 
For each $\tau\in[1-\eps,1]$ consider a  form  $\kappa_\tau:=\mu_\tau-\tau\dot\mu_\tau$ on $Q$ and choose a vector field $Y_\tau$ directing $\Ker\,\kappa_\tau$.
 
Consider a tangent to fibers $w\times R, w\in W$, vector field $  Y$ which is equal to $X_{s,t}$ on $w\times R$ for $w\in W_0$  and to $Y_\tau$ on $w\times R$  for $w=(u,\tau)\in C=\p W\times[1-\eps, 1]$.

 \begin{lemma}\label{lm:widehatX}
 We have 
\begin{itemize}
\item[1.] $X= Y+aZ$ over $W_0\times Q$, where the function $a:Q\to\R$ is determined by the equation 
\begin{align}\label{eq:function-a1}
a\mu=\iota(Y)d\mu.
\end{align}
\item[2.] $X=Y+aZ+bR$ over $C\times Q$, where the functions
  $a,b:Q\to\R$ are determined  by the equations
  \begin{equation}\label{lm:widehatX2}
  \begin{split}
 & a\kappa_\tau=\iota(Y_\tau)d\mu_\tau, \\
 & b\tau+\mu_\tau(Y_\tau)=0.
 \end{split}
 \end{equation}
 \end{itemize}
 \end{lemma}
 
\begin{remark}\label{rem:a} Note   that  for any vector $v$ and a  symplectic form $\omega$ we have
 $\iota(v)\omega(v)=0$. Hence,    the equation  \eqref{eq:function-a1}  and the first equation \eqref{lm:widehatX2} are always solvable for some function $a:Q\to\R$.
 \end{remark}

\begin{proof}
1. According to Remark \ref{rem:a} we can solve equation \eqref{eq:function-a1} with respect to a function $a$.
 We  have  $$\iota(aZ+ Y)d (\mu + \lambda) = \iota(Y)d\mu   +a\iota(Z) d \lambda =a(\mu+\lambda).$$
 But this means that $aZ+Y$ is tangent to the characteristic foliation defined by the form $\mu+\lambda=\wh\mu$.

\medskip 2.  
 Note that $d \wh\mu = d \mu_{\tau} + d (\tau \gamma)  = d \tau \wedge \dot{\mu_{\tau}} + d \mu_{\tau} + d \tau \wedge \gamma  +  \tau d \gamma.$ Hence, we get that 
\begin{align*} \iota(Y+aZ+bR)d \wh\mu =-(\dot\mu(Y_\tau)+b)d\tau+a\mu+a\tau\gamma=a(\mu+\tau\gamma),
\end{align*}
where 
we used the  second equation \eqref{lm:widehatX2} to conclude that $\dot\mu(Y_\tau)+b=0$.
But  this implies that $$\wh\mu\wedge  \left(\iota( Y+aZ+bR)d\wh\mu\right)=0,$$ which means that  $Y+aZ+bR$ generates the characteristic foliation on $C\times Q$ defined by the form $\wh\mu$, as required.
 
\end{proof}

\begin{proof}[Proof of Proposition \ref{prop:prelim-plug}]
(1a) When a   trajectory  of $X_1$ which enters at a point of  $ {}^{in}P$ it is in  the region where  $X_1=Y_1 + a Z$ with $a<0$. Since $a$ is negative the trajectory continues to remain in the region where the plug is given by $ W_0\times R,$  and hence projects onto trajectories of $-Z$ and $Y_1$, when projected to the corresponding factors. Moreover, the projection to $R$ remains in $R_-$, according to Property f) of Lemma \ref{lm:2d-plug}, 
 and therefore remains  in the region where the coefficient $a$ is negative.
 But in $R_-$ any trajectory of $Y_1$ entering at a point in $[\eps,1-\eps_1)\times 0$ converges to a negative  zero of $Y_1$, while every trajectory of $-Z$ converges to a zero of $Z$. Hence, any trajectory of $Y_1$ 
 entering through $ {}^{in}P$ converges to a $0$ of the vector field $X_1$.
 \medskip
 
 (1b)
  If a trajectory enters  at a  $(u, z)\in{}^{in}T  $ then similarly to 1a) we have  $X_1=Y_1+aZ$ for a negative coefficient $a$, and therefore, it projects    onto trajectories of $-Z$ and $Y_1$ in the factors $W$ and $R$ respectively. But every trajectory of $Y_1$ which enters through $(1-\eps_1,1]\times 0$ exits $R$ at a point of $[1-\frac\eps2,1]$, and hence the corresponding trajectory of $X_1$ exits ${}^{out}V$ at a point $(u'=Z^{-c}(u), z')$ for some positive $c>0$ and $z'\in(-\frac\eps2.1]$.
  \medskip
  
 (1c) and (1d) follows from the same arguments as, respectively, (1a) and (1b) applied to the vector field $-X_1$.
 
 \medskip
 Properties (2)-(4) are straightforward from the corresponding properties in Lemmas \ref{lm:2d-plug} and  \ref{lm:widehatX}.
 
(5) According to  Corollary \ref{cor:good-Lyapunov} it is sufficient to verify that
\begin{itemize}
\item each trajectory of $X_{s,t}$  either originates at $V\times 0$, or at a critical point of $X_{s,t}$;
\item  each trajectories of $X_{s,t}$  either terminates at $V\times 1$, or at a critical point of $X_{s,t}$;
\item $X_{s,t}$ has no retrograde connections.
\end{itemize}
In addition to the Weinstein subdomain $W_0=W\setminus\p W\times(1-\eps,1]$ consider also a larger subdomain  $W_1=W\setminus\p W\times(1-\frac\eps2,1]$.
Denote $Q_\pm:=W_1\times R_\pm,  Q_b:=(W\setminus W_1)\times R\subset Q.$  We have $Q=Q_+\cup Q_-\cup Q_b$.
Let us first analyze  the forward trajectory $X_{s,t}^u(p)=(w(u),r(u))$, $u\in\R_+$ of a  point 
 $p=(w,r)\in Q$. 
 
 If $p\in Q_-$ and $w\in W_0$ then $w(u)$  belongs to the negative Liouville trajectory $\bigcup\limits_{\tau\geq 0}Z^{-\tau}(w)$ as long as $r(u)\in R_-$.
 Similarly, if $w\in W_1\setminus W_0=\p W\times[1-\eps,1-\frac\eps2]$ the second coordinate of $w$ decreases as long as $r(u)\in R_-$.
But  the $R$-component $Y_{s,t}$  of $X_{s,t}$ is by construction inwardly transverse to the boundary of $R_-$, and hence,  remains in $R_-$ for all $u\geq 0$. Therefore, the trajectory either converges to a singular point of $X_{s,t}$, or exits through $V\times 1$.

Suppose $p\in Q_b$. Recall that the  $R$-component $Y_{s,t}$  of $X_{s,t}$ has in $Q_b$ a positive projection to  the  $y$-direction. Hence,  $X_{s,t}^u(p)$  either  exits through $V\times 1$, or enters $Q_+\cup Q_-$. But  in that case it only can enter $Q_-$, and therefore, the analysis of the previous case does apply.
 
 Suppose now that $p\in Q_+$.    If $w\in W_0$ then $w(u)$ moves along a positive Liouville trajectory of $w$, and if  $w\in \p W\times [1-\eps ,1-\frac\eps2]$,  then the second coordinate of $w(u)$ increases   as long as $r(u)\in R_+$. Hence, the trajectory either exits through $V\times 1$, or  enters either $Q_b$ or $Q_-$, and therefore, the previous analysis applies.
 
  Backward trajectories could be analyzed  similarly, with exchanging $Q_+$ and $Q_-$ cases.
 The absence of retrograde connections follows from (4).

 \end{proof}

 \subsection{Approximating balls by Weinstein cylinders}\label{sec:approx-W-cyl-balls}
 A hypersurface $\Sigma\subset(M,\xi)$ is said to have an {\em admissible corner}  along a smooth hypersurface $S\subset\Sigma$ if 
 \begin{itemize}
 \item $S$ is a codimension $2$ contact submanifold of $(M,\xi)$;
 \item $\Sigma\cap\Op S=\Sigma_1\cup\Sigma_2$, where $\Sigma_1$ and $\Sigma_2$ are two manifolds with common boundary $S$ which transversely intersect along $S$. 
 \end{itemize}
 We will  call the hypersurface $S\subset \Sigma$   the {\em corner locus} of $\Sigma$ and denoted by $\corner(\Sigma)$.

  Suppose $\wt \Sigma$ is a smooth hypersurface, let us choose its tubular $\eps$-neighborhood $N$ and denote by $\wh \tau$ the hyperplane field on $N$ orthogonal to the fibers of the projection $N\to\wt \Sigma.$
  We say that  a hypersurface with admissible corners  $ \Sigma\subset N$  
  is  {\em $C^1$  $\eps$-close to  $\wt\Sigma$}
 if  all its tangent planes    do not deviate for more than $\eps$ from $\wh\tau$.

 Given a Weinstein domain $(W,\lambda)$,  a domain $U$ in a
 contact manifold $(M,\xi)$ is called a {\em Weinstein cylinder}
 if there is given a  contactomorphism  $\phi:(W\times [0,a], \Ker(\lambda+dz))\to (U,\xi|_U)$.  We already encountered Weinstein cylinders in Proposition \ref{prop:prelim-plug}.

   Note that  the boundary of a Weinstein cylinder  $\phi(W\times[0,a])$ is a hypersurface with admissible corners along $\corner(U)=\p W\times 0  \cup \p W\times a$. We denote $\p_-U:=\phi(W\times 0),\; \p_+U:=\phi(W\times a)$.

For  a  general  $W$ the contact topology of the Weinstein cylinder $W\times[0,a]$ is very  sensitive to the value of the parameter $a$.  However, there is one exception (see e.g. \cite{EKP06}): 
 \begin{lemma}\label{lm:stand-Wein}
 Let $D=D^{2n-2}$ be the unit ball on $\R^{2n-2}$ endowed with the Liouville form $\lambda_\st:=\sum\limits_1^{n-1}(x_jdy_j-y_jdx_j)$.  Then for any $a>0$ there is a contactomorphism
 $$\Delta_a:=(D\times [0,a], \Ker(\lambda+dz))\to \Delta_1:=(D\times [0,1], \Ker(\lambda+dz)).$$
 \end{lemma}
Hence,  we will  use the notation $\Delta$   for any Weinstein cylinder of the type $D \times [0,a]. $

        \begin{lemma} \label{lm:approx-ball-by-cyl}
      Let $D$ be the standard contact ball and $p_\pm\in\p D$ its poles.
    Then for any $\eps>0$ there exists
  a contact embedding $h:\Delta\to D $ such that
    $h(\p_+\Delta)\subset \p D$, $\p D\setminus h(\p_+\Delta)$ is contained in   an $\eps$-neighborhood of  the pole $p_-$ and $D\setminus\Delta$ is contained in the $\eps$-neighborhood of $\p D$.
  \end{lemma}
  \begin{proof}
Consider a   $(2n-2)$-dimensional  open disc $B_-\subset\p D$ of radius $\eps$ centered at  $p_-$ with boundary $\p B_-$ transverse to the characteristic foliation $\ell_{\p D}$.  Denote $D_+:=\p D\setminus B_-$  By scaling the contact form along $\p D$   we can arrange that $D_+$ with the resulted form is the standard Liouville ball, and flowing for some time $\delta$ with the corresponding Reeb field $\Upsilon$, inwardly transverse to $D_+$, we  construct a  Weinstein cylinder $\Delta =\bigcup\limits_{t\in [0,\delta]}\Upsilon^t(D_+)\subset D$.  
  Let $D', D''\subset D$ be smaller standard  contact balls such that $D'\subset\Int  \Delta$, $D'\subset\Int D''$ and $D\setminus D''$ is in an $\eps$-neighborhood of $\p D$.
  Note that the space of contact embeddings of a standard contact ball into any connected contact manifold is connected. 
Hence, there exists a  contact diffeotopy $h_t:D\to D$, $t\in[0,1]$,  which is fixed on $\Op\p D$, and such that
$h_1(D')=D''$. Then the Weinstein cylinder  $h_1(\Delta)$
  has the required properties.
  \end{proof}

     \subsection{Weinstein cylinders in a good position}\label{sec:good-pos}

We say that Weinstein  cylinders $V_1,\dots, V_k$ are in a {\em good position}, see 
Fig.~\ref{fig:good_position}, if
  \begin{itemize}
  \item $V_j\cap V_{j+2}=\varnothing$ for  all $j=1,\dots, k-2$;
    \item $\p V_1\setminus V_2 \subset \Int\, \p_-V_1$;
  \item  for each $j=2,\dots, k$ we have $\p V_j\setminus  V_{j-1}\subset \Int\, \p_+V_j$;
\item  for each $j=2,\dots, k-q$ we have $\p V_j\setminus  V_{j+1}\subset \Int\, \p_+V_j$;
  \item $\p V_j$ and $\p V_{j+1}$,  $j=1,\dots,k-1$, intersect transversely along a codimension 2 contact submanifold  $S_j$, and the orientations  induced on $S_j$ from $\p V_{j+1}\setminus\Int V_j$ and from   $\p V_{j }\setminus\Int V_{j+1}$ are opposite.
  \end{itemize}
       
\begin{figure}[ht]
 
\includegraphics[scale=0.8]{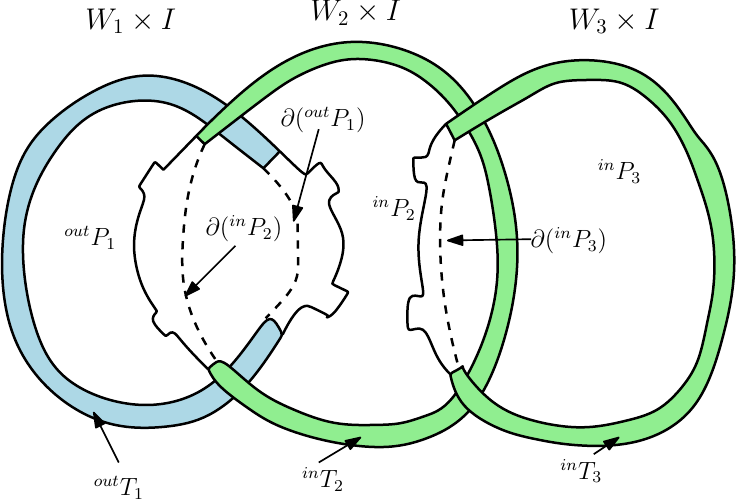}
 
\caption{Three Weinstein cylinders $V_1 = W_1 \times I,$ $V_2 = W_2 \times I,$ and $ V_3 = W_3 \times I$ in a good position.  The blue region  is 
$^{out}T_1$, while the green ones are
$^{out}T_2$  and $^{out}T_3$. }
\label{fig:good_position}
\end{figure}
  Note that if $V_1,\dots, V_k$ are in a good position then $\p\left(\bigcup\limits_1^k V_i\right)$ is a  piecewise smooth hypersurface with admissible corners, and it can be made Weinstein convex  by a $C^\infty$-small perturbation.

  \medskip
  
 Let $V\subset (M,\xi)$ be a domain diffeomorphic to a closed ball with a Weinstein convex boundary $\p V$. Let $S\subset \p V$ be a dividing set.  We say that $(V,S)$ {\em can be approximated by standard contact balls} if there exists a neighborhood $N_M$ of $S$ in $M$ such that for every $\sigma>0$ there is a (iso-)contact embedding $g:D\to (M,\xi)$ of the standard contact ball such that 
 \begin{itemize} 
 \item[-] $g(\p D)$ is contained in a $\sigma$-neighborhood of $\p V$;
 \item[-] $g(\p D)\cap N_M=\p V\cap N_M$.
 \end{itemize}
  \begin{prop}\label{prop:3-blocks} Let $V\subset (M,\xi)$ be a domain with a Weinstein convex boundary $\p V$ and  $S\subset \p V$  be a dividing set. Suppose that $(V,S)$ can be approximated by standard contact balls. Then for any $\eps>0$ there exist  three Weinstein cylinders $V_1, V_2, V_3\subset V$ in a good  position such that  a piecewise smooth hypersurface $\p(V_1\cup V_2\cup V_3)$ is $C^1$ $\eps$-close to $\p V$, see Fig.~\ref{fig:block_arrangement}.
 \end{prop}

\begin{figure}[ht]
 
\includegraphics[scale=0.8]{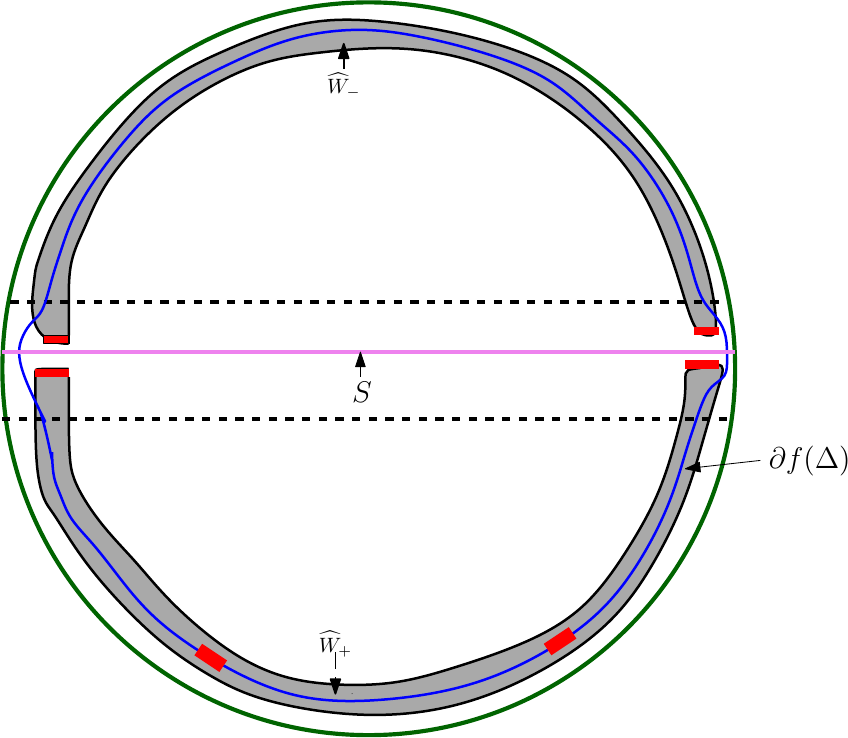}
 
 \caption{Three Weinstein cylinders   in good position approximating a  domain which admits an approximation by standard balls.} 
 
\label{fig:block_arrangement}
\end{figure}

 \begin{proof}
 Choose an inwardly pointing  transverse  to $\Sigma:=\p V$    contact vector field and  consider an interior collar $C:=\Sigma\times[0,1]\subset V$ such that $\Sigma\times 0=\Sigma$,  $x\times[0,1], x\in\Sigma$ are trajectories of $\Upsilon$ and $\Upsilon^t(\Sigma)=\Sigma\times t$. By scaling $X$ we can assume that $C$ is contained in  an $\eps$-small neighborhood of $\Sigma$.    
Choose an $\frac\eps4$-approximation of $(X^{\frac12}(V), X^{\frac12}(S))$ by a standard ball $ h(D)$.  By assumption the standard sphere $\wt\Sigma:=h(\p D)$ coincides with $\Sigma_{\frac12}:=\Sigma\times\frac12$ along a tubular neighborhood $N=S\times[-1,1]\subset\Sigma_{\frac12}$ contained in
a tubular neighborhood $N_M$  of $\Sigma_{\frac12}$ in $M.$ We can assume that $S\times t\subset N$ is transverse to the characteristic foliation on $\Sigma_{\frac12}$ for all $t\in[-1,1]$. Take a $C^\infty$-function
 $\theta:[0,1]\to\R$ which is equal to $1 -u $  for  $u\in [ \frac14,\frac34]$, equal to $0$ near $1$  and has non-positive derivative everywhere.
 Let $\Theta:\Sigma\to\R$  be a function supported in $N$ and defined
on $N$ by the formula  $\Theta(x,u)=\theta(|u|),\; (x,u)\in N=S\times[-1,1]. $
For   $\sigma\in(0,\frac12)$ define an isotopy $g_s:\Sigma\to V$:
$$g_s(w)=\Upsilon^{-s\Theta(u)}(u),\; u\in\Sigma, s\in[0,\sigma].$$

If $\sigma$ is chosen sufficiently small then the spheres  $g_s(\wt\Sigma)$, $s\in[0,\sigma]$, 
are almost standard, and hence, by Lemma \ref{lm:making-standard} a $C^0$-small adjustment near  one of the poles  makes them    standard. Assuming this is done we can extend  the isotopy to  a global compactly supported contact diffeotopy $G_s:V\to  V, s\in[0,\sigma].$

The dividing set $S\subset \Sigma_{\frac12}$ divides  $\Sigma$ into  domains $\Sigma_{\frac12,\pm}$ with the common boundary $S$. Denote
\begin{align*}
&\wh\Sigma_+:=\Sigma_{\frac12,+}\setminus \left(S\times\big(-\frac12,0\big]\right);\\
&\wh\Sigma_-:=\Sigma_{\frac12,-}\setminus \left(S\times\big[0,\frac12\big)\right).
\end{align*}
Consider two  Weinstein cylinders: $\wh W_\pm=\wh\Sigma_\pm\times[-\frac\sigma2,\frac\sigma2]=\bigcup\limits_{|s|\leq\frac\sigma2}\Upsilon^s(\wh\Sigma_\pm),$  see Fig. \ref{fig:block_arrangement}. We have
$\p_\pm \wh W_+=\Upsilon^{\mp\frac\sigma2}(\wh\Sigma_+)$ and
$\p_\pm \wh W_-=\Upsilon^{\pm\frac\sigma2}(\wh\Sigma_-)$. 
Consider the standard ball $\wh D:=G_\sigma(h(D))\subset V$.  The standard sphere $\p\wh D$ transversely intersects $\p \wh W_+$ along $\Upsilon^{\frac\sigma2}(S\times(-\frac12))\in\p_+\wh W_+$, and 
transversely intersects $\p \wh W_-$  along $\Upsilon^{-\frac\sigma2}(S\times\frac12)\in\p_-\wh W_-$. Note that the South pole of the sphere $\p \wh D$ is contained in $\Int\wh W_-$. Hence, using Lemma \ref{lm:approx-ball-by-cyl} we can  $C^1$-approximate   $\wh D$ by a contact embedding   $f:\Delta=D\times I\to \wh D$ such that $f(\corner(\p\Delta))$ is contained in $\Int\wh W_-$. It then follows that the Weinstein cylinders $\wh W_+,f(\Delta)$ and $\wh W_-$ are in good position, while the  boundary  $\wh W_+\cup f(\Delta)\cup\wh W_-$ \;$C^1$-approximates $\Sigma$.
    \end{proof}
 
     \subsection{Quasi-plugs}\label{sec:quasi-plug}

     Let $V\subset (M,\xi)$ be a domain whose boundary $\p U$  is a hypersurface with admissible corners. Denote  $Q:=V\times[0,a]$ and let $y$ be the coordinate corresponding to the second factor.
    Given  sufficiently small   $\sigma>0 $ and  a vector field $Y$ on  $Q$,  we call $(Q,Y)$ a {\em    $ \sigma  $-quasi-plug}   if the following conditions are satisfied:
\begin{itemize}
\item[QP1.] $Y$ coincides with $\frac{\p}{\p y}$ on $ \Op\p Q$;
\item[QP2.] $Y$ admits a Morse Lyapunov function  which is equal to $y$ on $\Op\p Q$;  
\item[QP3.] for  any point $p\in V$   with $\dist (p,\p V)>\sigma$ the  trajectory of $Y$    through     $ p\times 0$ converges  to a critical point of $Y$;  
 \item[QP4.] given any  point $p\in V$ with  $\dist (p,\p V)\leq\sigma$, there exists a point $p'\in\p V$ with $\dist(p,p') < \sigma$ and a positive $u(p)$ such that   the  trajectory of $Y$    through  a  point $ p\times 0 $  either converges  to a critical point of $Y$, or exit $Q$ at a point $ p''\times a $   with  $\dist (p'', L^{u(p)}p') <\sigma,$ where $L$ denotes a characteristic vectorfield for $\p V.$  
\end{itemize}
 We will use quasi-plugs in combination with the following simple observation.
\begin{lemma}
\label{lm:remove-quasi} Suppose that the characteristic foliation on $\p V$ is $\sigma$-short. Then any $ \sigma $-quasi-plug is a $3\sigma$-plug. 
\end{lemma}
    \begin{lemma}\label{lm:good-positions-cor}
    Let $V_1,\dots, V_k$, $k\geq 2$  be $k$ Weinstein cylinders in a good position.   
  Denote $\wh V=\bigcup\limits_1^k V_j$, $\wh Q:=\wh V\times[0,k]$,
  $ Q_j:=V_j \times[j-1,j] , \; j=1,\dots,k$.  Denote by $Y_j$   the  vector field $X_1$ constructed in Proposition \ref{prop:prelim-plug} and implanted into $ Q_j$, $ j=1,\dots, k$.  Let  $Y$ be  the resulting  field on $\wh Q$.  Then $(\wh Q,Y)$ is a  $C \sigma $-quasi-plug for some constant $C$.                                                                                                                                                                                                                                                                                                                                                                                                                                                                                                            
     \end{lemma} 
     \begin{proof}  Recall the notations ${}^{in}P, {}^{out}P, {}^{in}T, {}^{out}T, {}^{in}S, {}^{out}S$ introduced in 
     Proposition \ref{prop:prelim-plug} for a plug $Q=V\times[0,1]$. As defined in \ref{prop:prelim-plug}, we have
      ${}^{in}P,  {}^{in}T,  {}^{in}S \subset V\times 0$ and   ${}^{out}P,  {}^{out}T,  {}^{out}S \subset V\times 1.$
      Let us introduce analogous domains
  ${}^{in}P_j,    {}^{in} T_j ,      {}^{in} S_j ,    {}^{out}P_j,  {}^{out} T_j,  {}^{out}S_j$ for plugs $Q_j= V_j\times[j-1,j]$, $j=1,\dots, k$.
     It will be, however, more convenient for us to view these domains as subdomains of $V_j$.  Denote  $\wh V_{\leq i}=\bigcup\limits_1^{i}V_j$,
     $\wh Q_{\leq i}:= \wh V_{\leq i}\times[0,i]$. 
         We will prove the  following more precise  statement by induction in   $i$: 
 
   \smallskip  {\em   $(\wh Q_{\leq i}, Y|_{ \wh Q_{\leq i}})$ is a $C\sigma$-quasi-plug for any $i\geq 2$. Moreover, any non-blocked (i.e. not converging  in any direction to a zero of $Y$)   trajectory $\gamma$  enters through $$\left({}^{in}S_1\cup( { }^{in}T_2\setminus V_1)\cup\dots\cup(
   {}^{in}T_i\setminus V_{i-1})\right)\times 0$$ and exits through  $$\left(({}^{out}T_1\setminus V_2)\cup( { }^{out}S_2\setminus (V_1\cup V_3))\cup( { }^{out}S_3\setminus (V_2\cup V_4))\cup\dots\cup({}^{out}S_i\setminus V_{i-1})\right)\times i.$$
  }
   
    Suppose  $i=2$.  
   For a sufficiently small  $\sigma$  we have   $ V_1\setminus ({}^{out}T_1\cup{}^{out}P_1) \subset  {}^{in}P_2 $.
   Hence,  a non-blocked trajectory which enters $V_1\times 0$ either  intersects   ${}^{out}T^1\times 1$, or enters $Q_2$  through $^{in}P_2$, and hence, is  blocked inside $Q_2$. If $\gamma\cap  (V_{\leq 2}\times 1)=\{(x,1)\subset {}^{out}T^1\times 1 $, then $\gamma\cap  (V_{\leq 2}\times 0)=\{(x',0)\}$, where $x'\in{}^{in}S_1$. Moreover, $x$ is moved from $x'$ in the positive direction of the Liouville flow of $W_1$ (or equivalently, in the positive direction of the characteristic foliation on $\p V_1$),   possibly with a $\sigma$-error. If $x\in  {}^{out}T^1\setminus V_2$ then $\gamma$ exists $Q_2$ at a point $(x,2)$. Otherwise $x\in {}^{in}T_2\cup{}^{in}P_2$. Hence, if the trajectory $\gamma$ is not blocked in $Q_2$ it exists at a point $(x'',2)\in \left({}^{out}S_2\right)\times 2$, where $x''$ is moved from $x$ in the negative direction of the Liouville flow of $W_2$ (or equivalently, in the {\em positive} direction of the characteristic foliation on $\p V_2$),   possibly with a $\sigma$-error.

    Suppose now that the statement holds for  $(\wh Q_{\leq i-1}, Y|_{\wh Q_{\leq i-1}})$.   By the induction assumption  a   trajectory $\gamma$ which  enters through $\wh V_{\leq i-1}\times 0 $ and is not blocked by $\wh Q_{\leq i-1}$  enters at a point $(x',0)$, $x'\in\left({}^{in}S_1\cup( { }^{in}T_2\setminus V_1)\cup\dots\cup(
   {}^{in}T_{i-1}\setminus V_{i-2})\right)$  and exits the plug $\wh Q_{\leq i-1}$  at a point
   $(x, i-1)$, $x\in  ({}^{out}T_1\setminus V_2)\cup( { }^{out}S_2\setminus (V_1\cup V_3))\cup( { }^{out}S_3\setminus (V_2\cup V_4))\cup\dots\cup({}^{out}S_{i-1}\setminus V_{i-2})$. Moreover, up to a $C\sigma$-error, $x$ moved from $x'$ in the positive direction of the characteristic foliation of $\p\wh V_{\leq i-1}$.  If $x\notin  V_i$ then $\gamma$ exits $Q_i$ at a point $(x,i)$. Otherwise $x\in {}^{in}T_i\cup{}^{in}P_i$. Hence, if the trajectory $\gamma$ is not blocked in $Q_i$ it exits at a point $(x'',i)\in \left({}^{out}S_i\right)\times i$, where $x''$ is moved from $x$ in the negative direction of the Liouville flow of $W_i$ (or equivalently, in the {\em positive} direction of the characteristic foliation on $\p V_i$),   possibly with a $\sigma$-error. 
     \end{proof}
     \begin{remark} For the purposes of this paper we will use Lemma \ref{lm:good-positions-cor}  in Section \ref{sec:proof-main} only for the case $k=3$ for  3 Weinstein cylinders in good position constructed in Proposition \ref{prop:3-blocks}.
     \end{remark}

 \section{From a quasi-plug to  a $\sigma$-plug}\label{sec:main_plug}
 As we already mentioned above, we prove Proposition \ref{prop:main-plug} by induction on dimension.
 Lemma \ref{lm:2d-plug} (together with  the height  reduction argument from Section \ref{sec:plug-height}, see also Remark \ref{rem:2d})  serves as  the base of the induction for  $2n=2$.
 Suppose that Proposition \ref{prop:main-plug} is already proven in dimension $<2n$.

\subsection{Standard and almost standard spheres in a contact manifold}
 Recall that we defined the standard contact $(2n+1)$-ball as  the upper hemisphere $D:=S^{2n+1}_+=\{y_{n+1}\geq 0\}$ in the unit  sphere  $ S^{2n+1}=\left\{\;\sum\limits_1^{n+1}x_j^2+y_j^2=1\right\}\subset \R^{2n+2}$ endowed with the standard contact structure $\xi_\st=\{ \sum\limits_1^{n+1} x_jdy_j-y_jdx_j =0\}.$

We call a germ of a contact structure  on $S^{2n}=\p D$ {\em standard} if it is contactomorphic to the germ  of $\xi_\st$ along $\p D$ and coincides with $\xi_\st$ on a neighborhood of poles $$p_\pm:=\{x_{n+1}=\pm 1,\; x_j=0\;\hbox{for}\;j=1,\dots, n,\; y_j=0\;\hbox{for}\; j=1,\dots, n+1\}.$$

 A germ $\xi$ along a sphere is called  {\em almost standard} if it coincides with $\xi_\st$ on a neighborhood of poles $p_\pm$ and its is   characteristic foliation on $S^{2n}$   admits a Lyapunov  Morse function with  exactly 2 critical points.  Recall (see Proposition \ref{prop:AG}) that a germ of a contact structure along a hypersurface $\Sigma$ is determined  by its restriction to $\Sigma$  up to a diffeomorphism fixed on $\Sigma$. Hence, we will not distinguish below between the germs and their restrictions.
 
 Recall that any linear (conformal) symplectic structure $\om$ on a vector space $E$ defines a canonical contact structure $\zeta_\om$
  on its sphere at infinity $S(E)$, i.e. the space of oriented lines through the origin.  The group $Sp(E,\om)$ of linear symplectic transformations acts by linear projective contactomorphisms on $(S(E),\zeta_\om)$  and this representation is faithful. Hence, we can view  $Sp(E,\om)$ as a subgroup of the group of contactomorphisms of $(S(E),\zeta_\om)$.
 In particular, given a hypersurface $\Sigma$ in a contact manifold $(M,\xi),$  the contact structure $\xi$ defines a conformal symplectic structure on $T_p\Sigma$ for each singularity $p$ of the characteristic foliation $\ell_{\Sigma,\xi}$. We denote the corresponding contact structure on the sphere ( $S_p=S(T_p\Sigma)$ at infinity)  by $\zeta_{p,\xi}$.
 
 For the  contact structure $\xi_\st $ along $S^{2n}=\p D$ the holonomy along the leaves of the characteristic foliation $\ell_{\xi_\st}$ allows us to identify the contact spheres at infinity $(S_\pm,\zeta_\pm):=\left(S \left(T_{p_\pm}(S^{2n})\right),\zeta_{p_\pm}\right).$ In turn, the contact sphere $(S_+,\zeta_+)$ can be canonically identified with the  standard contact sphere $(S^{2n-1},\xi_\st)$
 Hence, for any almost standard germ $\xi$ along $S^{2n}$ the holonomy along the leaves
 of $\ell_\xi$ can be viewed as a contactomorphism $h_\xi:(S^{2n-1},\xi_\st)\to (S^{2n-1},\xi_\st)$. 
 We will call $h_\xi$ the {\em clutching contactomorphism} of an almost standard germ $\xi$.
 Let $AlSt$  be the space of almost standard contact germs on $S^{2n}$  and $\D$   the group of contactomorphisms  of  the standard contact sphere $(S^{2n-1},\xi_\st)$.  The image of the map $\pi: AlSt\to\D$ is the subgroup $\D_0\subset\D$ which consists of  contactomorphisms which are pseudo-isotopic  to the identity.  Denote by $St$ the subspace of $AlSt$ which consists of standard contact germs.  For  $\xi\in St$  we have $h_\xi\in \D_1:=Sp(\R^{2n},\om_\st)\subset \D_0$. 
  The following lemma is straightforward.
\begin{lemma}\label{lm:al-st-holonomy} The projection  $\pi:AlSt\to\D_0$  is a Serre fibration. If $h_\xi\in Sp(\R^{2n},\om_\st)$ then $\xi$ is standard, i.e. $St=\pi^{-1}(\D_1)$.  
\end{lemma}

     \begin{figure}[ht]
 
\includegraphics[scale=0.85]{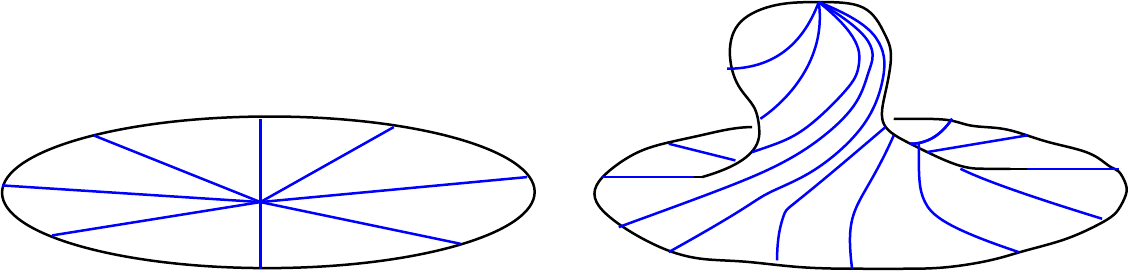}
 
\caption{Adjusting a characteristic foliation near a pole.}
\label{fig:adjust_near_pole}
\end{figure}

 \begin{lemma}\label{lm:making-standard}
 Let $\xi_t$ be a family  of  contact structures on $\Op\p D\subset(S^{2n+1},\xi_\st)$ such that their  germs along $\p D$ are almost standard. Suppose that for $t\in[0,\frac18]$ the germ $\xi_t$ is standard.
 Then there exists a diffeotopy $g_t:  \Op\p D\to\Op\p D$ supported in an arbitrarily  small neighborhood of $p_+$ such that    the germs $g_t^*\xi_t$ along $\p D$ are standard and $g_0$ is the identity.  See Fig. \ref{fig:adjust_near_pole}.
      \end{lemma}
 \begin{proof}  Recall that  by assumption $\xi_t=\xi_0=\xi_\st$ in a neighborhood  $U\supset \{p_+\}, U\in\Op\p D$. There exists a smaller neighborhood $U'\subset U$ such that
 the pair $(U',\p D\cap U';\xi_\st)$ is contactomorphic to $(D^{2n}\times(-\eps,\eps), D^{2n}\times0; \Ker(\gamma:=\sum\limits_1^n x_i dy_i-y_idx_i +dz)).$
 Denote $u:=\sum\limits_1^n(x_j^2+y_j^2)$, and consider the splitting $D^{2n}\setminus 0=S^{2n-1}\times(0,1]$, given by the radial projection to the unit sphere and the $u$-coordinate.
 With respect to this splitting the form $\gamma$ can be  written as $dz+u\alpha_\st$, where $\alpha_\st$ is the standard contact form of the standard contact $(2n-1)$-sphere.
 
    \begin{figure}[ht]
 
\includegraphics[scale=0.8]{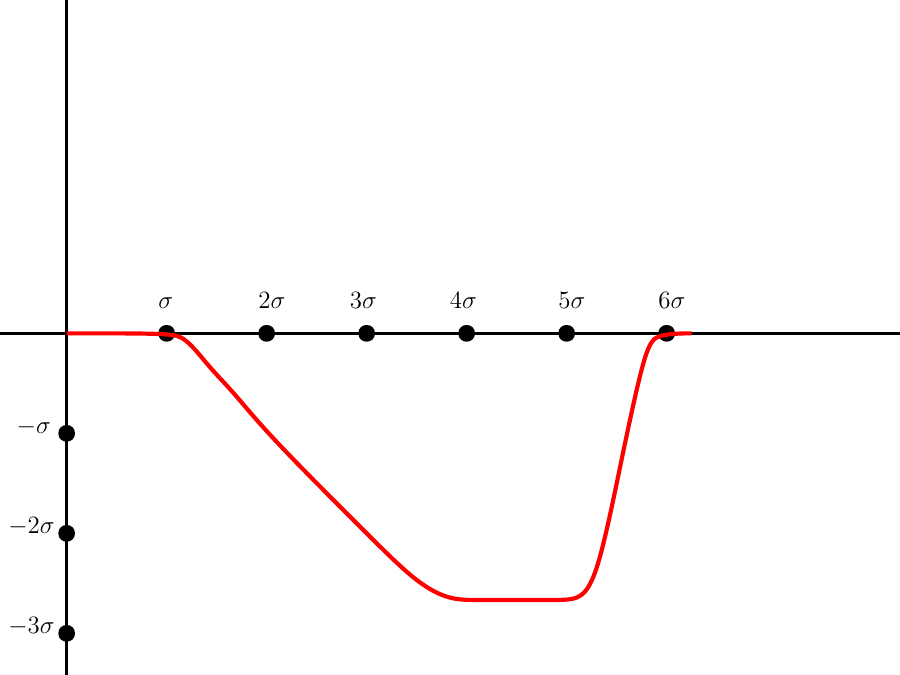}
 
\caption{The function $\theta(u)$.}
\label{fig:theta}
\end{figure}
 Choose a positive  $  \sigma\ll\eps,$ and consider  a  $C^\infty$-function $\theta:[0,\infty)\to[-3\sigma,0]$, see \ref{fig:theta} which is supported in $[0,6\sigma]  $, and such that $\theta(u)=0$ for $u<\sigma$, $\theta(u)=-3\sigma$ for $u\in[4\sigma, 5\sigma], $ $\theta(u)= \sigma -u$ for $u\in[2\sigma,3\sigma]$ and $\theta'(u)\leq 0$ for $u\in[0,5\sigma]$. Let us view $u\in[0,1],z\in(-\eps,\eps)$ and $w\in S^{2n-1}$ as coordinates in the neighborhood $U'$, so that the equation $z=0$ defines $\p D\cap U'$. Consider the family of  hypersurfaces $\Theta_{s}\subset \Op\p D$ which coincides with graphs  $\{z=s\theta(u)\}$ in  $U'$ and equal to $\p D$ elsewhere. Take a neighborhood  $U'':=\{u<7\sigma,
   |z|<\sigma\}\Subset U'$.  There exists a   supported  in $U'' $ diffeotopy $\phi_s:\Op\p D\to\Op \p D$  such that $\phi_0=\Id$,  $\phi_s(\p D)=\Theta_s$. Define  a   diffeotopy $\psi_t:\Op\p D\to\Op\p D$   as   $\psi_t:=\phi_{8t}$ for $t\in[0,\frac18]$, $\psi_t=\phi_1$ for $t\in[\frac18,\frac78]$ and $\psi_t=\phi_{8-8t}$ for $t\in[\frac78,1]$.
     Note that the germs  of  contact structures  $\wt\xi_{t}:=\psi_t^*\xi_t$ along $ \p D$ are almost standard, and moreover, for each $t\in[0,1]$ the  clutching diffeomorphism   $\wt h_t:=h_{\wt\xi_{t}}$ differs from $h_{\xi_t}$  by a unitary rotation  $w\mapsto e^{C(t)i\phi}w$ of the sphere $S^{2n-1}\subset \C^n$.
    Consider the domain   $\wh U:=\{u\leq 3\sigma,\; -2\sigma\leq z\leq-\sigma\}\subset U''$. Note  that for $t\in[\frac18,\frac78]$ we have $\psi_t(\p D) \cap\wh U=\{ u=\sigma-z,\;-2\sigma\leq z\leq-\sigma\}$.   Consider the space $\cK$ of functions    $K:[-2\sigma,-\sigma]\times S^{2n-1}\to (0,3\sigma]$ which are  equal to $\sigma-z$ near the boundary. Given $K\in\cK$   consider its graph $\Gamma_K:= \{u=K(z,w);\; (z,w)\in  [-2\sigma,-\sigma]\times S^{2n-1}\}\subset \wh U$. The contact structures $\xi_t$ in $\wh U$ is given by the form $\frac1udz+\alpha_\st$ and hence, the holonomy along the leaves of the characteristic foliation $\ell_{\Gamma_K}$ is equal to the time $\sigma$ map of the contact flow    of the contact Hamiltonian $\frac1{K(z,w)}$. We view  here $z\in[-2\sigma,-\sigma]$ as the time parameter. 
    
    Define a  contact isotopy $g_t:=\wt h_t^{-1}=(h_{\wt\xi_t})^{-1}:S^{2n-1}\to S^{2n-1}$. While its contact Hamiltonian $G_t:S^{2n-1}\to S^{2n-1}$   is not necessarily positive, it can be made positive and  even arbitrarily large by composing $g_t$ with appropriate unitary  rotations $w\mapsto e^{\wt C(t)i\phi}$ of the sphere
   $S^{2n-1}\subset \C^n$. We  will keep the notation   $g_t$  for the modified isotopy.  Hence, there exists a family  of functions $K_t\in \cK$, $t\in[\frac18,\frac78]$  such that  $K_t=\sigma-z$ for $t=\frac18,\frac78,$  and  such that the holonomy along the leaves of the characteristic foliation $\ell_{\Gamma_{K_t}}$  coincide with $g_t$ up to a unitary rotation of the sphere $S^{2n-1}$. Let us modify the diffeotopy $\psi_t$ for $t\in[\frac18,\frac78]$, keeping it supported in $\wh U$, so that $\psi_t(\p D\cap\wh U)=\Gamma_{K_t}$.
  Denote   $\wh\xi_t:=\psi_t^*\xi_t$. By construction the clutching diffeomorphisms $h_{\wh\xi_t}:S^{2n-1}\to S^{2n-1}$ are unitary rotations, and hence, the germs of contact structures $\wh\xi_t$ along $\p D$ are standard.
  
 \end{proof}

  \subsection{Making the characteristic foliation short} \label{sec:making-short}

 \begin{prop}\label{prop:plugs-in-convex} Suppose that Proposition \ref{prop:main-plug}  holds for plug installation over $D^{2n-3}\times[0,1]$.
 Let $\Sigma$ be the standard $(2n-2)$-dimensional sphere in  $(M^{2n-1},\xi=\Ker\alpha)$. Then  for any $\sigma>0$ there exists  an $\sigma$-small in the $C^0$-sense isotopy $f_{s}:\Sigma\to M$ starting  with  the inclusion $f_0:\Sigma\hookrightarrow M$ such that
  \begin{itemize} 
 \item[a)] $f_{s}$ is fixed on a neighborhood of poles of the characteristic foliation $\ell$  of $\Sigma$;
 \item[b)] the family of characteristic foliations $\ell_{s}$, $s\in[0,1]$, induced on $\Sigma$ by $f_s^*\alpha$;
  admits a family of good Lyapunov functions $F_s:\Sigma\to\R$;
  \item[c)] the characteristic foliation $\ell_{1}$  
 is $\sigma$-short;
 \item[d)] for any $\sigma>0$ the isotopy $f_{s}$ can be included into a 2-parametric isotopy $f_{s,t}, \;s,t\in[0,1]$, such that
 \begin{itemize}
 \item[(i)] $f_{s,0}=f_s$, $f_{0,t}=f_0$ for all $s,t\in[0,1]$;
 \item[(ii)] the  spheres $f_{s,1}(\Sigma)$ are almost standard for all $s\in[0,1]$;
 \item[(iii)] $f_{s,t}$ is         $\sigma$-close to $f_{s,0}$ for all $(s,t)\in[0,1]$;
 \item[(iv)]  the isotopy $f_{1,t}, t\in[0,1]$, is fixed on a neighborhood of a dividing set $S_1$ of  $\Sigma_1=f_1(\Sigma)$.
 \end{itemize}
 \end{itemize}
  \end{prop}
 
  \begin{proof}

   Using Corollary \ref{cor:good-Lyapunov} we can find  an  $\sigma$-blocking  system  $\{D_j\}  $, $j=1,\dots, N$,   of transverse  standard contact  discs. We can assume that  
   the Lyapunov function $F$ for $X$ on $\Sigma$ is constant on each $D_j$. Denote $c_j=F|_{D_j}$ and assume   that $c_1< c_2<\dots< c_N$.
   Let $  Q_j\subset\Sigma $ be  disjoint flow-boxes of $D_j$ of diameter $<2\sigma$. 
 Let $h_{s,t}:  Q:=D\times[0,1]\to D\times T^*[0,1]$ be a $C^0$ small isotopy constructed in Proposition \ref{prop:main-plug}.  
We define the required isotopy $f_{s,t}$ by successively deforming flow-boxes $ Q_j$. Denote $\Delta_j:=[\frac{j-1}N,\frac{j}N]\subset[0,1]$, $1\leq j\leq N$.
For $  s\in\Delta_j, t\in[0,1]$ 
  define
 $f_{s,t}:=\Phi_j\circ h_{Ns-j+1,t}\circ \Phi_j^{-1}$ on $  Q_j$, 
 $f_{s,t}= \Phi_i\circ h_{1,t}\circ \Phi_i^{-1}$ on $\  Q_i$ for $i<j$, and fixed elsewhere on $\Sigma$. Then Proposition \ref{prop:main-plug} guarantees  all  the  required properties  of the isotopy $f_{s,t}$ except b), d)(ii) and d)(iv). More precisely, for b) we automatically get a family of Lyapunov functions for $X_s$, but not necessarily good ones. Moreover, we get a family of Lyapunov functions for the whole 2-parametric family  $X_{s,t}$. In particular, Lyapunov functions for $X_{s,1}$ have no critical points except the  maximum and the minimum of the  original Lyapunov function $F$. This implies that the sphere $f_{s,1}(\Sigma)$ are almost standard.   Recall that according to Corollary \ref{cor:Lyapunov-convexity-param} it is sufficient to  ensure properties (L1) and (L2) and absence  of retrograde connections.
    According to property c)(v) of Proposition \ref{prop:main-plug} trajectories of $X_{s,t}$ converging to positive zeroes enter each plug  through the same $(n-2)$-dimensional stratified subset $E_j\subset D_j\subset \p\wh Q_j$.  According to our staged construction of the isotopy, the isotopy $f_{s,0}$ for $s\in[\frac{j-1}N,\frac jN]$, $j=1,\dots, N$, which creates a plug in $\wh Q_j$  does not change  trajectories of $X_s=X_{s,0}$ in $F\leq c_j$. Hence, by a $C^\infty$-small adjustment of embeddings $\phi_j$ before each step of the isotopy we can arrange that the closure   $\overline{G_j}$  of the negative tail $\bigcup\limits_{u\geq 0}X_{\frac{j-1}N}(E_j)$ of the set $E_j$ does not contain any negative zeroes of $X_{\frac{j-1}N}$, and hence,  the same property holds for all $s\geq \frac{j-1}N$. The deformation $f_{s,t}$ for a fixed $s\in\Delta_j$ changes the field $X_{s,t}$ only in an arbitrarily small neighborhood of $\ol G_j$, and hence thanks to compactness of the set of zeroes, one can arrange that $X_{s,t}$ have no retrograde connections. Corollary \ref{cor:good-Lyapunov} then guaranteed that $f_{s,t}(\Sigma)$ are Weinstein convex. It remains to satisfy property
d)(iv). Consider the dividing set $S$ for $X_1$. Using  Property c)(iii) of Proposition \ref{prop:main-plug} we can find an isotopy of $\Sigma$ preserving leaves of $\ell_1$ with disjoins $S$ with  compact subsets $C_1^\pm$, and hence with their neighborhoods $U^\pm\supset C_1^\pm$.  According to  c)(ii)  we can arrange the isotopy $f_{1,t}$      to be supported in $U^\pm$, which implies property d)(iv).

  \end{proof}
  
 \begin{prop}\label{prop:making-short-approx} Suppose that Proposition \ref{prop:main-plug}  holds for plug installation over $D^{2n-3}\times[0,1]$. Let $D=(D^{2n-1},\Ker\alpha_\st)$ be the standard contact disc. Then for any $\sigma>0$ there exists a  $\sigma$-small in the  $C^0$-sense  isotopy $h_s:D\to D$     such that $h_s(D)\subset\Int D$ for all $s>0$ and 
 \begin{itemize}
 \item[(i)] the ball $\wt D:=h_1(D)$ has a Weinstein convex boundary  $\p\wt D$ with the dividing set $S\subset\p\wt D$;
 \item[(ii)] the characteristic foliation $\ell_{\p\wt D}$ is $\sigma$-short;
 \item[(iii)] $(\wt D,S)$ can be approximated by standard contact balls.
 \end{itemize}
 \end{prop}
 \begin{remark} \label{rem:tight}
 In the case $n=2$ property (iii) is automatic from the classification of tight contact structures on the $3$-ball, see \cite{El-tight}.
 \end{remark}

\begin{proof}
 Let us first shrink $D\to \Int D$ by a $C^\infty$-small contracting contact isotopy and then apply to the image $D'$   Proposition \ref{prop:plugs-in-convex}. Let us extend the constructed there isotopy   $f_{s,t}:\p D'\to D$, $s,t\in[0,1]$,  to an isotopy $D'\to D$. We will continue using the notation $f_{s,t}$ for the extension.  We claim that that the isotopy obtained by concatenating the shrinking isotopy with the isotopy $f_{s,0}$ has the required properties. Indeed, the balls 
 $D'_s:=f_{s,0}(D')$ have Weinstein convex boundaries with dividing sets $S_s\subset \p D'_s$  and the characteristic foliation on  $\p D'_1$ is   $\sigma$-short 
 and the family of spheres $\Sigma_s:=f_{s,1}(\p D')$, $s\in[0,1]$ are almost standard  and hence, can be made standard by an arbitrarily $C^0$-small adjustment away from the poles and dividing sets by applying 
  Lemma \ref{lm:making-standard}. The sphere $\Sigma_0=\p D'$ bounds the standard ball, and hence, the same holds, by continuation of the contact isotopy argument, for $\Sigma_1$. But by construction $\Sigma_1$ coincides with $\p D'_1$ on a neighborhood of the dividing set $S_1\subset \p D'_1$ and  can be made arbitrarily $C^0$-close to $\p D'_1$, i.e. $(D'_1,S_1)$ can be approximated by standard contact balls.
\end{proof}

   \subsection{Proofs on main results}   \label{sec:proof-main}

  \begin{proof} [Proof of Proposition \ref{prop:main-plug-weak}]
   Applying the induction hypothesis  and Proposition \ref{prop:making-short-approx}  we find a disc $\wt D\subset D$ such that
    \begin{itemize}
 \item[-]  $\wt D$ has a Weinstein convex boundary  $\p\wh D$ with the dividing set $S\subset\p\wh D$;
 \item[-] the characteristic foliation $\ell_{\p\wt D}$ is $\sigma$-short;
 \item[-] $(\wt D,S)$ can be approximated by standard contact balls;
 \item[-]  $D\setminus f(\p D)$ is contained in $\sigma$-small neighborhood of $\p D$,
 \end{itemize}
         Next, we apply Proposition \ref{prop:3-blocks} and find $3$ Weinstein cylinders $V_1, V_2, V_3\subset f(D)$  such that  a piecewise smooth hypersurface $\p(V_1\cup V_2\cup V_3)$ is $C^1$ $\sigma$-close to $\p f(D)$. Applying now Lemma \ref{lm:good-positions-cor} we  install into $D\times[0,1]$ a $\sigma$-quasi-plug in $(V_1\cup V_2\cup V_3)\times[0,1]\subset f(D)\times[0,1]$. But the characteristic foliation on $f(\p D)$ is $\sigma$-short, and hence the constructed plug  is a genuine $C\sigma$-plug for $D\times [0,1]$ for some universal constant $C>0$.
 
 This concludes the proof    of   Proposition \ref{prop:main-plug-weak}, and hence, of Proposition \ref{prop:main-plug}.
\end{proof}
 \begin{proof}[Proof of Theorem \ref{thm:main-precise}]
 First, we adjust $\Sigma$ by  a  $C^\infty$-isotopy to make all singularities of its characteristic foliation $\ell_\Sigma$ non-degenerate and hyperbolic. Next we apply Lemma~\ref{lm:exist-discs-cont} to find a blocking collection of transverse standard contact discs $D_j\subset \Sigma$. According to Lemma~\ref{lm:sigma-plug} there exists  $\sigma>0$  such that by  installing $\sigma$-plugs instead of flow boxes one  can arrange  the resulting flow to satisfy condition (L1).
 Proposition \ref{prop:main-plug} asserts that such plugs can be installed by deforming flow-boxes via an arbitrarily small in the $C^0$-sense isotopy.   By an additional $C^\infty$-perturbation of the hypersurface outside plugs we can satisfy the  Morse-Smale property, while still preserving condition (L1), see Lemma~\ref{lm:sigma-plug}, and
hence,  by Corollary  \ref{cor:good-Lyapunov}   the resulting $\Sigma$ is Weinstein convex.
 This concludes the proof  of  Theorem ~\ref{thm:main-precise},  and  in combination with  
Lemma~\ref{lm:W-convex}   of Theorem~\ref{thm:HH}.
 \end{proof}
\bibliographystyle{amsplain}

\begin{thebibliography}{100}
\bibitem{Arn73} V. I.~Arnold,   Ordinary Differential Equations,
  MIT Press, Cambridge, Massachusetts (1973). 
\bibitem{AG}   V. I. Arnold and A. B. Givental, Symplectic Geometry, in  {\em Dynamical Systems IV}, Encylop. of Math. Sciences, Springer-Verlag, 1990, 4--138.
\bibitem{Bo} J. M. Boardman, Singularity of differentiable maps. {\em Publ. math. IHES}, {\b 33} (1967), 21--57.
\bibitem{CE12} K. Cieliebak and Y. Eliashberg, From Stein to Weinstein and back, {\em Colloquium Publications},  AMS,
 {\bf 59}(2012).
 \bibitem{El-tight} Y. Eliashberg, {Contact 3-manifolds 20 years since J.~Martinet's work},
 {\sl Ann. Inst. Fourier}, {\bf 42}(1992), 165--192.
 \bibitem{El17} Y. Eliashberg, Weinstein manifolds revisited,   {\it Proc. Sympos. Pure Math.}, {\bf 99}(2018),   59--82. 
  \bibitem{EG91} Y. Eliashberg and M. Gromov, Convex symplectic manifolds, {\em Proc. 
Sympos. Pure Math.},
({\bf 52})(1991), Part 2, 135--162.
\bibitem{EKP06} Y. Eliashberg, S.-S. Kim and L. Polterovich,  Geometry of contact domains and transformations: orderability vs.
   squeezing,   {\it Geometry and Topology}, {\bf 10}(2006) 1635--1747.
\bibitem{Gi} E. Giroux, Convexit\'e en topologie de contact, {\em Comm. Math. Helvet.}, {\bf 66}(1991), 637–677.
\bibitem{Gi00} E. Giroux,  Structures de contact en dimension trois et bifurcations des feuilletages de surfaces. \textit{Invent. Math} Vol. 144, no. 3,  (2000), 
615--689.
\bibitem{Gro86} M. Gromov, {\em Partial Differential Relations}, { Springer-Verlag}, Berlin-Heidelberg, 1986.
\bibitem{GE71} M. Gromov and Y. Eliashberg, Construction of a smooth mapping with the prescribed
Jacobian, {\em Funct. anal. i pril.,} {\bf 14} (1980),
  89--90
\bibitem{Ho00} K. Honda, On the classification of tight contact structures. I, {\em Geom. Topol.} 4 (2000), 309--368.  
\bibitem{HH} K. Honda and Y. Huang, Convex surface theory in contact topology,  arXiv:1907.06025v2 (math.SG).
\bibitem{Morin} B. Morin, Formes canoniques des singularit\'es d’une application diff\'erentiable, {\em C. R. Acad.
Sci. Paris}, ({\bf 260})1965, 5662-5665, 6503-6506.
\bibitem{DS} D.\,\,Salamon, Notes on hypersurfaces in contact manifolds, preprint 2022, https://people.math.ethz.ch/\%7Esalamon/PREPRINTS/CONVEX.pdf.
\bibitem{Thom} R. Thom, Les singularit\'es des applications différentiables, {\em Ann.de l'Inst. Fourier}
 {\bf 6}(1956), 43--87.
\bibitem{We05} M. Weiss,   What does the classifying space of a category classifies?, {\em Homol., homot, and applic.}, {\bf 7}(2005), 185--195.

 \bibitem{Wi}  W. Wilson Jr. On minimal set of non-singular vector fields. \textit{Annals of Mathematics}. {\bf 84}(1996), 529--536.
\end{thebibliography}

\end{document}